\documentclass[11pt]{article}

\usepackage[margin=1cm]{geometry}
\usepackage{fullpage}
\usepackage{amsmath}
\usepackage{amsthm}
\usepackage{amssymb}
\usepackage{mathtools}
\usepackage{enumerate}
\usepackage{textcomp}
\usepackage[hidelinks]{hyperref}
\usepackage{cleveref}

\newcommand{\mbb}[1]{\mathbb{#1}}
\newcommand{\mbf}[1]{\mathbf{#1}}
\newcommand{\mc}[1]{\mathcal{#1}}

\newcommand{\bs}{\boldsymbol}
\newcommand{\tr}{\textup{tr}\,}
\newcommand{\wt}[1]{\widetilde{#1}}
\newcommand{\wh}[1]{\widehat{#1}}
\newcommand{\diag}{\textup{diag}\,}
\newcommand{\pf}{\textup{pf}\,}
\newcommand{\Tr}[1]{\bigl\langle{#1}\bigr\rangle}
\newcommand{\diff}{\,\mathrm{d}}
\renewcommand{\det}{\textup{det}}
\renewcommand{\Re}{\textup{Re}}
\renewcommand{\Im}{\textup{Im}}
\numberwithin{equation}{section}

\newtheorem{theorem}{Theorem}[section]

\newtheorem{lemma}{Lemma}[section]
\newtheorem{corollary}{Corollary}[section]
\newtheorem{proposition}{Proposition}[section]

\theoremstyle{remark}

\theoremstyle{definition}
\newtheorem{definition}{Definition}[section]

\title{Precise Delocalisation and Gumbel Laws for Eigenvectors of Wigner Matrices}
\author{Mohammed Osman\footnote{mohammed.osman@ist.ac.at}}
\date{\small Institute of Science and Technology Austria}

\begin{document}
\maketitle

\abstract{We prove a delocalisation bound for eigenvectors of Wigner matrices with the precise relationship between the size of the largest entry and the decay exponent of the probability. We also prove that the largest entry of an individual eigenvector and the largest entry of all eigenvectors are both Gumbel distributed. The proof is based on the inclusion-exclusion principle, the small probability comparison of Erd\H{o}s--Xu, and partial diagonalisation of Gaussian-divisible matrices.}

\section{Introduction}
Let $V$ be a random $N\times N$ Hermitian matrix distributed according Gaussian Orthogonal/Unitary Ensemble (GOE/GUE). Due to unitary invariance, the eigenvalues and eigenvectors of $V$ are independent, and the matrix of eigenvectors is uniformly distributed on $\textup{O}(N)$ ($\beta=1$) or $\textup{U}(N)$ ($\beta=2$). In particular, individual eigenvectors are uniformly distributed on the sphere. It is well-known that such random vectors behave like a vector of independent standard Gaussians. More precisely, if $\mbf{u}$ is uniform on the sphere and $\mbf{x}$ is a vector of independent standard Gaussians, then a result of Diaconis--Freedman \cite{diaconis_dozen_1987} shows that
\begin{align}
	d_{TV}(\mbf{u}_{I},\mbf{x}_{I})&\simeq\frac{|I|}{N},
\end{align}
where $\mbf{a}_{I}$ denotes the subvector with components in $I\subset[N]$. Thus, any $o(N)$ subvector is close in total variation distance to a vector of independent Gaussians. Later work by Jiang \cite{jiang_how_2006} extended this to a matrix $U$ uniformly distributed on $\textup{O}(N)$ or $\textup{U}(N)$. In this case, one can simultaneously approximate submatrices of size $p\times q$, where $pq=o(N)$.

The approximation by Gaussians also holds on the level of the largest entry. One manifestation of this is delocalisation: for any $D>0$ there is an $\ell^{(\beta)}_{N}(D)$ (see \eqref{eq:l} below) such that
\begin{align}
	\mbb{P}\Bigl(\|\mbf{u}\|_{\infty}>\sqrt{\frac{\ell^{(\beta)}(D)}{N}}\Bigr)&\simeq N^{-D},\label{eq:deloc}
\end{align}
Moreover, the largest entry $\|\mbf{u}\|_{\infty}$ has Gumbel fluctuations, i.e.
\begin{align}
	\lim_{N\to\infty}\mbb{P}\Bigl(N\|\mbf{u}\|_{\infty}^{2}-\ell^{(\beta)}_{N}(0)\leq x\Bigr)&=e^{-e^{\frac{\beta x}{2}}}.
\end{align}
Jiang \cite{jiang_maxima_2005} proves the analogous matrix result: if $U$ is uniform on $\textup{O}(N)$ or $\textup{U}(N)$, then
\begin{align}
	\lim_{N\to\infty}\mbb{P}\Bigl(N\max_{i,j\in[N]}|U_{i,j}|^{2}-\ell^{(\beta)}_{N}(1)\leq x\Bigr)&=e^{-e^{\frac{\beta x}{2}}}.\label{eq:jiang}
\end{align}

Following the general principle of universality in random matrix theory, these results are expected to hold for a broad class of matrices beyond the GOE/GUE. The class we are interested in is the class of Wigner matrices, defined as follows.
\begin{definition}
A random Hermitian matrix $W$ is a \textbf{Wigner matrix} if its upper-triangular entries are independent random variables satisfying:
\begin{align}
	\mbb{E}W_{ij}&=0,\quad\mbb{E}|W_{ij}|^{2}=\frac{1}{N},\quad\mbb{E}|W_{ij}|^{p}\leq \frac{C_{p}}{N^{p/2}}.
\end{align}
\end{definition}
Eigenvectors of Wigner matrices are by now very well understood. Erd\H{o}s--Schlein--Yau \cite{erdos_local_2009} prove a weaker form of \eqref{eq:deloc} as a consequence of a local law for the resolvent: if $\mbf{u}_{n}$ are the eigenvectors of $W$, then for any $\epsilon>0$ and $D>0$ we have
\begin{align}
	\mbb{P}\Bigl(\|\mbf{u}_{n}\|_{\infty}>\sqrt{\frac{N^{\epsilon}}{N}}\Bigr)&\leq N^{-D}.
\end{align}
The $N^{\epsilon}$ term was improved to $C(D)\log N$ for some $C(D)$ by Benigni--Lopatto \cite{benigni_optimal_2022}. The latter work also studies small deviations and proves that, for any $\epsilon>0$, there is a $\delta>0$ such that
\begin{align}
	\mbb{P}\Bigl(\|\mbf{u}_{n}\|_{\infty}>\sqrt{\frac{(2/\beta+\epsilon)\log N}{N}}\Bigr)&\leq N^{-\delta}.
\end{align}

In another direction, finite collections of entries of eigenvectors have been shown to converge to independent Gaussians (in the sense of moments) by Bourgade--Yau \cite{bourgade_eigenvector_2017}. Recently, Benigni \cite{benigni_quantitative_2026} has improved this result to allow for the simultaneous approximation of a growing number of entries, though not quite at the threshold  in Jiang's result. 

Finally, a recent work of Bao--Bucht--Schnelli \cite{bao_order_2026} obtains the convergence in distribution of order statistics of edge eigenvectors, i.e. those with index $n<N^{\epsilon}$ or $N-n>N^{\epsilon}$ for sufficiently small $\epsilon>0$. In particular, the maximum entry is shown to be Gumbel distributed. We remark that the Gumbel distribution at the edge can also be obtained from the work of Benigni--Lopatto \cite{benigni_optimal_2022} by direct comparison with the GOE/GUE, since universality at the edge holds under two-moment matching (see the paper of Knowles--Yin \cite{knowles_eigenvector_2013}). 

In this paper we extend the upper bound in \eqref{eq:deloc} to Wigner matrices, with an extra (suboptimal) factor of $\log^{C} N$ on the right hand side. In addition, we prove that the largest entries of individual eigenvectors are Gumbel distributed in the bulk and edge, in the sense of fixed energy. This means that we fix a point in the support of the limiting spectrum and look at the eigenvector whose eigenvalue is closed to that point. We also prove Gumbel fluctuations for the maximum over all eigenvectors. In other words, we extend \eqref{eq:jiang} to the matrix of eigenvectors of a Wigner matrix in the case $\beta=2$. This last result holds in full generality for complex Wigner matrices, while in the real case we require an additional four-moment matching with the GOE. In the next subsection we give a precise statement of these results.

\subsection{Main Results}

We use the common convention that $\beta=1$ stands for real matrices and $\beta=2$ for complex matrices. Define
\begin{align}
	l_{N}^{(\beta)}(D)&:=\begin{cases}
	2(D+1)\log N-\log(4(D+1)\pi\log N)&\quad \beta=1\\
	(D+1)\log N&\quad\beta=2
	\end{cases}.\label{eq:l}
\end{align}
Let $\rho_{sc}(x)$ be the density of the semicircle law, i.e.
\begin{align}
    \rho_{sc}(x)&=\frac{\sqrt{4-x^{2}}}{2\pi}.
\end{align}
For $E\in[-2,2]$, let 
\begin{align}
    d_{E}&=\min\Bigl((N\rho_{sc}(E))^{-1},N^{-2/3}\Bigr),
\end{align} 
denote the typical eigenvalue spacing at $E$. Our first main result is delocalisation of eigenvectors with the precise relationship between the size of $\|\mbf{u}_{n}\|_{\infty}$ and the decay exponent of the probability. 
\begin{theorem}[Precise delocalisation]\label{thm1}
Let $D>1$, $E\in[-2,2]$ and $I$ be an interval centred at $E$ such that $|I|>N^{-D}$. Let 
\begin{align}
	\rho_{sc}(I)&:=\int_{I}\rho_{sc}(\lambda)\diff\lambda.
\end{align}
There is a constant $C>0$ such that 
\begin{align}
	\mbb{P}\left(\max_{n:\lambda_{n}\in I}\|\mbf{u}_{n}\|_{\infty}>\sqrt{\frac{\ell_{N}^{(\beta)}(D)}{N}}\right)&\leq CN^{1-D}\rho_{sc}(I),\label{eq:thm11}
\end{align}
If either i) $\beta=1$ and $|I|\lesssim d_{E}$; or ii) $\beta=2$, we also have the lower bound
\begin{align}
	\mbb{P}\left(\max_{n:\lambda_{n}\in I}\|\mbf{u}_{n}\|_{\infty}>\sqrt{\frac{\ell_{N}^{(\beta)}(D)}{N}}\right)&\geq\frac{1}{C}N^{1-D}\rho_{sc}(I).\label{eq:thm12}
\end{align}
Moreover, we have
\begin{align}
	\max_{n\in[N]}\mbb{P}\left(\|\mbf{u}_{n}\|_{\infty}>\sqrt{\frac{\ell_{N}^{(\beta)}(D)}{N}}\right)&\leq C(\log^{C} N)N^{-D},\label{eq:thm13}
\end{align}
\end{theorem}
We remark that the restriction $|I|\lesssim d_{E}$ in the lower bound for real matrices is purely technical and the result should hold uniformly in $|I|\gtrsim N^{-D}$. The logarithmic factor in \eqref{eq:thm12} is suboptimal and arises because we can only localise a given eigenvalue $\lambda_{n}$ to an interval of size $(N^{2/3}\hat{n}^{1/3})^{-1}\log^{C} N$ with very high probability, where $\hat{n}=\textup{min}(n,N+1-n)$. In the bulk this can be improved to $\log N$ using the optimal rigidity result of Bourgade--Lopatto--Zeitouni \cite{bourgade_optimal_2025}.

The next two results deal with the limiting distributions of $\|\mbf{u}_{n}\|_{\infty}$ and $\max_{n}\|\mbf{u}_{n}\|_{\infty}$. Firstly, we obtain Gumbel fluctuations for the largest entry of individual eigenvectors. We phrase it in terms of the eigenvalue closest to a given point rather than the eigenvalue with a given index. This is because our method for treating Gaussian-divisible matrices can only work with unlabelled eigenvalues.
\begin{theorem}[Gumbel laws for individual eigenvectors]\label{thm2}
Fix $E\in[-2,2]$ and $x\in\mbb{R}$. Let $n\equiv n(E)$ be the (random) index of the eigenvalue closest to $E$. Then we have
\begin{align}
	\lim_{N\to\infty}\mbb{P}\left(N\|\mbf{u}_{n}\|^{2}_{\infty}-\ell_{N}^{(\beta)}(0)\leq x\right)&=e^{-e^{-\frac{\beta x}{2}}}.
\end{align}
\end{theorem}
The deterministic index version of this result at the edge (i.e. indices $n\leq N^{\epsilon}$ or $N-n\geq N^{\epsilon}$ for small $\epsilon$) is obtained as a special case of a general result about order statistics by Bao--Bucht--Schnelli \cite{bao_order_2026}. It can also be deduced from the results of Benigni--Lopatto \cite{benigni_optimal_2022}, since only two moments are needed to match at the edge and hence one can directly compare a Wigner matrix with the GOE/GUE. Note that the random index $n(E)$ has fluctuations of size $\log N$, and so we cannot deduce the deterministic index result from the fixed energy result (and vice versa).

Secondly, we obtain the Gumbel law for the maximum entry of all eigenvectors, i.e. the analogue of Jiang's result for $\max_{i,j\in[N]}|U_{ij}|$ when $U$ is uniform in $\textup{U}(N)$ or $\textup{O}(N)$. This result is proven in general only for complex Wigner matrices; in the real case we require four-moment matching.
\begin{theorem}[Gumbel law for matrix of eigenvectors]\label{thm3}
Let $W$ be either:
\begin{enumerate}[i)]
\item a complex Wigner matrix;
\item a real Wigner matrix whose first four moments match the GOE.
\end{enumerate}
Then
\begin{align}
	\lim_{N\to\infty}\mbb{P}\left(\max_{n\in[N]}N\|\mbf{u}_{n}\|_{\infty}^{2}-\ell_{N}^{(\beta)}(1)\leq x\right)&=e^{-e^{-\frac{\beta x}{2}}}.
\end{align}
\end{theorem}
The general method applies to both real and complex matrices, but the obstacle to removing the moment-matching condition in Theorem \ref{thm3} for real matrices is our lack of good asymptotics for
\begin{align}
	\mbb{E}_{Y}\Bigl[\prod_{j}|\det(X+\sqrt{t}Y-\lambda_{j})|\Bigr]
\end{align}
when $|\lambda_{i}-\lambda_{j}|\gg N^{-2/3}\hat{i}^{-1/3}$. Here $X$ is a deterministic matrix and the expectation is with respect to $Y\sim GOE(N)$.

\subsection{Proof Outline}
The proof of each theorem splits into two parts: i) comparison of Wigner matrices by moment-matching; ii) computations for Gaussian-divisible matrices. For the comparison part, we start from the eigenvector regularisation used by Benigni--Lopatto:
\begin{align}
	\frac{1}{\pi}\int_{I_{\delta}(\hat{\lambda}_{n})}\mbf{q}^{*}\Im G(E+i\eta)\mbf{q}\diff E,
\end{align}
where $I_{\delta}(x)$ is an interval of size $N^{-1-\delta}$ centred at $x$ and $\hat{\lambda}_{n}$ is a regularisation of the eigenvalue $\lambda_{n}$. As shown by Benigni--Lopatto, this is a good approximation of $|\mbf{u}_{n}^{*}\mbf{q}|^{2}$ on the event in which $I_{\delta}(\hat{\lambda}_{n})$ contains only $\lambda_{n}$, which holds with probability $1-O(N^{-\delta})$. This observation is enough for the comparison step of Theorem \ref{thm2}, since this is a local result and so we can work on the intersection of $O(N^{o(1)})$ such events. However, Theorem \ref{thm3} is a global result and we cannot simultaneously exclude other eigenvalues from intervals around each eigenvalue. Moreover, the probability of the bad event is only $O(N^{-\delta})$ and hence not small enough to obtain Theorem \ref{thm1}.

Our solution is to use the inclusion-exclusion principle to reduce the problem to the comparison of
\begin{align}
	\mbb{P}\Bigl(N|u_{i,j}|^{2}>x,\,i\in I,j\in J\Bigr),
\end{align}
for arbitrary fixed subsets $I,J\subset [N]$. Since $|I|$ is fixed, this probability only involves a fixed number of eigenvectors and we can reasonably hope for an approximation of the form
\begin{align}
	\mbb{P}\Bigl(N|u_{i,j}|^{2}>x,\,i\in I,j\in J\Bigr)&\simeq\mbb{P}\Bigl(\{N|u_{i,j}|^{2}>x,\,i\in I,\,j\in J\}\cap\mc{G}(I)\Bigr),
\end{align}
where $\mc{G}(I)$ is the ``good" event that small intervals around $\lambda_{i}$ exclude other eigenvalues, simultaneously for each $i\in I$. If the eigenvalues and eigenvectors were independent (as they are in the GOE/GUE), such an approximation would be trivial. In the absence of this independence, we first compare 
\begin{align}
	\mbb{P}\Bigl(\{N|u_{i,j}|^{2}>x,\,i\in I,\,j\in J\}\cap\mc{G}^{c}(I)\Bigr)
\end{align}
with the Gaussian-divisible case, for which we can show an approximate independence between eigenvalues and eigenvectors. In order to carry out these steps, we must compare probabilities of size $N^{-k}$ for any finite $k$, which we do using the iterative scheme of Erd\H{o}s--Xu \cite{erdos_small_2023}.

For the Gaussian-divisible part, we use the method of partial diagonalisation, which is the Hermitian analogue of the partial Schur decomposition (see \cite{maltsev_bulk_2024}). This allows us to derive explicit integral formulas for expectations of functions of unlabelled eigenvalues and eigenvectors. In our case we need to consider indicator functions
\begin{align*}
	\mbf{1}_{(x,\infty)}(N\|\mbf{u}_{i}\|_{\infty}^{2}-\ell^{(\beta)}_{N}(D)).
\end{align*}
Since we are dealing with low probability events, we can only afford additive errors of $O(N^{-D})$ for some large $D$, and thus require a suitable integrable representation that allows for multiplicative errors. This is provided by the distributional identity
\begin{align*}
	1_{(x,\infty)}(y)&=\frac{1}{2\pi i}\int_{-\infty}^{\infty}\frac{e^{ik(y-x)}}{k-i0}\diff k.
\end{align*}
By approximating $\mbb{E}[e^{iNk|u_{i,j}|^{2}}]$, we show that we can replace entries of $\mbf{u}_{i}$ with independent Gaussian random variables that are also independent of the eigenvalues in such expectation values. 

\section{Preliminaries}
In this section we collect the relevant previous results. We start with the familiar local law. Recall that $\rho_{sc}(x)$ denotes the density of the semicircle law and let 
\begin{align}
    m_{sc}(z)&=\frac{\sqrt{z^{2}-4}-z}{2}
\end{align}
denote its Stieltjes transform.
\begin{lemma}[\cite{erdos_local_2009,erdos_local_2013}]\label{lem:locallaw}
Let $D>0$ and $\epsilon,\xi>0$. With probability at least $1-N^{-D}$ we have
\begin{align}
    |\Tr{G(z)}-m_{sc}(z)|&\leq\frac{N^{\xi}}{N|\Im z|},\\
    \max_{i,j\in[N]}|G_{ij}(z)-m_{sc}(z)\delta_{ij}|&\prec\frac{N^{\xi}}{\sqrt{N|\Im z|}},
\end{align}
uniformly in $|\Im z|>N^{-1}$, and
\begin{align}
    |\Tr{G(z)}-m_{sc}(z)|&\leq\frac{N^{\xi}}{N(|\Re z|-2+|\Im z|)},\\
    \max_{i,j\in[N]}|G_{ij}(z)-m_{sc}(z)\delta_{ij}|&\leq\frac{N^{\xi}}{N^{1/2}(|\Re z|-2+|\Im z|)^{1/4}},
\end{align}
uniformly in $|\Re z|-2>N^{-2/3+\epsilon}$.
\end{lemma}
Here and in the rest of the paper $\Tr{A}:=N^{-1}\tr A$.

Now we come to some definitions, following Benigni--Lopatto \cite{benigni_optimal_2022}. For an index $i\in[N]$ we define 
\begin{align}
	\hat{i}&:=\textup{min}(i,N+1-i).
\end{align}
For $w\in\mbb{C}$, let 
\begin{align}
	\theta^{(a,b)}_{w}(W)&:=W+(w-1)W_{a,b}\mbf{e}_{a}\mbf{e}_{b}^{*},
\end{align}
be the matrix whose $(a,b)$ and $(b,a)$-entries are replaced by $wW_{a,b}$ and $wW_{b,a}$ respectively. The following result from \cite{benigni_optimal_2022} gives an approximation of the eigenvalues in terms of resolvents.
\begin{lemma}[Proposition 4.1 in \cite{benigni_optimal_2022}]\label{lem:eigapprox}
Fix $\delta,\epsilon,D>0$. For each $i\in[N]$ there exists a function $\hat{\lambda}_{i}\equiv\hat{\lambda}_{i,\delta,\epsilon}:\textup{Mat}_{N}\to\mbb{R}$ satisfying the following:
\begin{enumerate}[i)]
\item with probability at least $1-N^{-D}$,
\begin{align}
	|\hat{\lambda}_{i}(W)-\lambda_{i}(W)|&\leq\frac{N^{\epsilon-\delta}}{N^{2/3}\hat{i}^{1/3}};
\end{align}
\item with probability at least $1-N^{-D}$,
\begin{align}
	\sup_{0\leq w\leq1}|\partial^{j}_{a,b}\hat{\lambda}_{i}(\Theta_{w}^{(c,d)}(W))|\leq\frac{N^{j(\epsilon+\delta)}}{N^{2/3}\hat{i}^{1/3}};
\end{align}
\item there is a $C>0$ such that
\begin{align}
	\sup_{0\leq w\leq 1}|\partial^{j}_{a,b}\hat{\lambda}_{i}(\Theta^{(c,d)}_{w}(W))|&\leq CN^{Cj}.
\end{align}
\end{enumerate}
\end{lemma}
For the remainder of this paper we fix $\delta>\xi>0$ and set $\hat{\lambda}_{i}\equiv\hat{\lambda}_{i,\delta+3\xi,\xi}$, so that
\begin{align}
	|\hat{\lambda}_{i}-\lambda_{i}|&\leq\frac{N^{-\delta-2\xi}}{N^{2/3}\hat{i}^{1/3}},
\end{align}
with probability at least $1-N^{-D}$. 

Let
\begin{align}
	I_{i}&:=\Biggl[\lambda_{i}-\frac{N^{-\delta}}{N^{2/3}\hat{i}^{1/3}},\lambda_{i}+\frac{N^{-\delta}}{N^{2/3}\hat{i}^{1/3}}\Biggr].\label{eq:I_i}
\end{align}
Let $\hat{I}_{i}$ be defined similarly, with $\hat{\lambda}_{i}$ in place of $\lambda_{i}$. Note that by Lemma \ref{lem:eigapprox}, $\lambda_{i}\in\hat{I}_{i}$ with probability at least $1-N^{-D}$. We define two notions of approximate eigenvector:
\begin{align}
	\hat{w}^{2}_{i}(\mbf{q})&:=\frac{1}{\pi}\int_{I_{i}}\mbf{q}^{*}\Im G(E+i\eta_{i})\mbf{q}\diff E,
\end{align}
and
\begin{align}
	\hat{v}^{2}_{i}(\mbf{q})&:=\frac{1}{\pi}\int_{\hat{I}_{i}}\mbf{q}^{*}\Im G(E+i\eta_{i})\mbf{q}\diff E,
\end{align}
where 
\begin{align}
	\eta_{i}&=\frac{N^{-\delta-\xi}}{N^{2/3}\hat{i}^{1/3}}.\label{eq:eta_i}
\end{align}
In other words, for $\hat{w}^{2}_{i}$ we centre around the true eigenvalue, while for $\hat{v}^{2}_{i}$ we centre around the approximate eigenvalue. We use the shorthand
\begin{align}
	\hat{w}^{2}_{i,j}&:=\hat{w}^{2}_{i}(\mbf{e}_{j}),\quad\hat{v}^{2}_{i,j}:=\hat{v}^{2}_{i}(\mbf{e}_{j}),
\end{align}
where $\mbf{e}_{j},\,j\in[N]$ are the coordinate vectors. We record some properties of these approximations (see \cite[Lemmas 4.10-4.12]{benigni_optimal_2022}).
\begin{lemma}\label{lem:eigapprox2}
For any $D>0$ and unit vector $\mbf{q}$, the following statements hold with probability at least $1-N^{-D}$:
\begin{enumerate}[i)]
\item
\begin{align}
	\hat{w}^{2}_{i}(\mbf{q})&=\hat{v}^{2}_{i}(\mbf{q})+O\Bigl(\frac{1}{N^{1+\xi}}\Bigr);
\end{align}
\item
\begin{align}
	\hat{v}^{2}_{i}(\mbf{q})&\leq\sum_{k:\lambda_{k}\in \hat{I}_{i}}|\mbf{u}_{i}^{*}\mbf{q}|^{2}+O\Bigl(\frac{1}{N^{1+\xi}}\Bigr);
\end{align}
\item
\begin{align}
	|\mbf{u}^{*}_{i}\mbf{q}|^{2}&\leq\hat{v}^{2}_{i}(\mbf{q})+O\Bigl(\frac{1}{N^{1+\xi}}\Bigr).
\end{align}
\end{enumerate}
\end{lemma}

We define $\mc{E}\equiv\mc{E}(\delta,\xi,D)$ to be the event that the statements in Lemmas \ref{lem:eigapprox} and \ref{lem:eigapprox2} hold.

\section{Moment Matching}
In this section we carry out the moment-matching arguments to compare a general Wigner matrix with a Gaussian-divisible one, where Gaussian-divisible means a matrix of the form $(1+t)^{-1/2}(X+\sqrt{t}Y)$ with $Y\sim GOE(N)/GUE(N)$. Throughout this section, we set $t=N^{-1/3+\rho}$ with $\rho>\delta$ and choose $\wt{W}$ to be a Gaussian-divisible matrix that $t$-matches with $W$, by which we mean that
\begin{align}
	\mbb{E}\Re W_{ij}^{p}\Im W_{ij}^{q}&=\mbb{E}\Re\wt{W}_{ij}^{p}\Im\wt{W}_{ij}^{q},\qquad p+q\leq 3,\\
	\mbb{E}\Re W_{ij}^{4-p}\Im W_{ij}^{p}&=\mbb{E}\Re\wt{W}_{ij}^{4-p}\Im\wt{W}_{ij}^{q}+O\Bigl(\frac{t}{N^{2}}\Bigr),\qquad p\leq 4.
\end{align}
For $E\in[-2,2]$, $m\in\mbb{N}$ and $f\in C^{\infty}_{c}(\mbb{R}^{2m})$, define
\begin{align}
	\mc{L}_{m,f,E}(W)&:=\sum_{i_{1},...,i_{m}}f\Bigl(\frac{\lambda_{i_{1}}-E}{d_{E}},...,\frac{\lambda_{i_{m}}-E}{d_{E}},N\|\mbf{u}_{i_{1}}\|_{\infty}^{2}-\ell^{(\beta)}_{N}(0),...,N\|\mbf{u}_{i_{m}}\|_{\infty}^{2}-\ell^{(\beta)}_{N}(0)\Bigr),
\end{align}
where the sum is over distinct indices in $[N]$ and we recall that $d_{E}$ is the typical spacing at $E$. The following lemma is the comparison step for Theorem \ref{thm1}. Since its proof is standard we only give a sketch (see e.g. Benigni--Lopatto \cite[Proof of Theorem 1.3]{benigni_optimal_2022} for a similar result).
\begin{lemma}\label{lem:momentmatching1}
Let $E\in[-2,2]$ and $m\in\mbb{N}$. Let $f\in C^{\infty}_{c}(\mbb{R}^{2m})$. Then there is a $\tau>0$ such that
\begin{align}
	\Bigl|\mbb{E}\bigl[\mc{L}_{m,f,E}(W)-\mc{L}_{m,f,E}(\wt{W})\bigr]\Bigr|&\leq N^{-\tau}
\end{align}
\end{lemma}
\begin{proof}[Proof sketch]
Since $f$ has compact support, for any $D>0$ and $\xi>0$, with probability $1-N^{-D}$ there are at most $N^{\xi}$ terms in the sum. By the level repulsion estimate \cite[Proposition 5.7]{benigni_optimal_2022}, we have
\begin{align}
    \mbb{P}\bigl(\mc{N}(\hat{I}_{i})\geq2\bigr)\leq N^{-\delta},    
\end{align}
where $\mc{N}(I)$ is the number of eigenvalues in $I$. We can therefore insert a smoothed indicator function of the event
\begin{align}
    \cap_{a=1}^{m}\bigl\{\mc{N}(\hat{I}_{i_{a}})=1\bigr\},
\end{align}
on which we can replace $\|\mbf{u}_{i}\|_{\infty}^{2}$ with $\max_{j}\hat{v}_{i,j}^{2}$. Now we use the smoothed maximum $\beta^{-1}\log\sum_{i}e^{\beta x_{i}}$ and argue by Lindeberg replacement.
\end{proof}

For Theorems \ref{thm2} and \ref{thm3}, we need a stronger comparison result that allows us to compare probabilities of order $N^{-k}$ for any $k$. This is done using an idea of Erd\H{o}s--Xu \cite{erdos_small_2023}. The difference here is that we use the discrete Lindeberg interpolation rather than the continuous Ornstein--Uhlenbeck interpolation since we work in the bulk and edge (at the edge the continuous interpolation is sufficient).

Let $D\geq0$, $m\in\mbb{N}$, $k_{a}\in\mbb{N},\,a=1,...,m$ and $J_{a}\subset[N],\,a=1,...,m$. Define
\begin{align}
	f_{m,\mbf{k}}(\lambda_{1},...,\lambda_{m})&=\prod_{a=1}^{m}\mbf{1}_{k_{a}+1}(\mc{N}(\hat{I}_{a})),
\end{align}
where $\mbf{1}_{k}\equiv\mbf{1}_{\{k\}}$, and
\begin{align}
	g_{m,D,J}(\lambda_{i_{1}},...,\lambda_{i_{m}})&:=\prod_{a=1}^{m}\prod_{j\in J_{a}}\mbf{1}_{(x,\infty)}(N\hat{v}_{i_{a},j}^{2}-\ell^{(\beta)}_{N}(D)).
\end{align}
The function $f_{m,\mbf{k}}$ restricts to the event
\begin{align}
	\bigcap_{a=1}^{m}\bigl\{\mc{N}(\hat{I}_{a})=k_{a}+1\bigr\},
\end{align}
in which there are $k_{a}+1$ eigenvalues in the interval $\hat{I}_{a}$ for $a=1,...,m$ (recall that $\hat{I}_{a}$ is centred at the approximate eigenvalue $\hat{\lambda}_{a}$ and $\lambda_{a}\in\hat{I}_{a}$ with probability at least $1-N^{-D}$).

We will bound the following sum of joint probabilities of eigenvalues and eigenvectors:
\begin{align}
	\mc{L}_{m,x,\mbf{k},D,J}(W)&:=\mbb{E}\sum_{i_{1},...,i_{m}}f_{m}(\lambda_{i_{1}},...,\lambda_{i_{m}})g_{m,D,J}(\lambda_{i_{1}},...,\lambda_{i_{m}}).
\end{align}
A key input, to be proven in the next section, is a bound on $\mc{L}_{m,x,\mbf{k},D,J}$ for Gaussian-divisible matrices $W=X+\sqrt{t}Y$ when at least one $k_{a}$ is positive.
\begin{proposition}\label{prop:p-gauss-divisible}
Let $t\gg N^{-1/3+\delta}$ and $W=X+\sqrt{t}Y$ be a real or complex Gaussian-divisible matrix if $m=1$ and a complex Gaussian-divisible matrix if $m\geq2$. Let $C>0$, $x\in[-C,C]$ and $\max_{a\in[m]}|J_{a}|\leq C$. Then there is a $c>0$ such that
\begin{align}
	\mc{L}_{m,x,\mbf{k},DJ}(W)&\lesssim N^{m-(D+1)\sum_{a=1}^{m}|J_{a}|-c\delta\sum_{a}k_{a}^{2}}.
\end{align}
\end{proposition}
The main point here is that the event $\{\mc{N}(\hat{I}_{a})=k_{a}+1\}$ for $k_{a}\geq1$ contributes a factor of $N^{-c\delta k_{a}^{2}}$ due to eigenvalue repulsion. In the case $m\geq2$, we are not able to prove this bound for real Gaussian-divisible matrices for technical reasons, although the result should still be true.

For the comparison, we first replace the indicators by smooth functions. Let $\alpha>0$ and $\chi_{x,\alpha}$ be a smooth approximation to $\mbf{1}_{(x,\infty)}$ on the scale $N^{-\alpha}$, so that
\begin{align}
	\chi_{x+N^{-\alpha},\alpha}&\leq\mbf{1}_{(x,\infty)}\leq\chi_{x-N^{-\alpha},\alpha}.
\end{align}
Let $\psi_{l}$ be a smooth approximation to $\mbf{1}_{(l-1/4,l+1/4)}$ so that (since $\mc{N}(I)$ is integer-valued)
\begin{align}
	\mbf{1}_{l}(\mc{N}(\hat{I}_{a}))&=\psi_{l}(\mc{N}(\hat{I}_{a})).
\end{align}
Finally, we approximate the eigenvalue counting function $\mc{N}$ by
\begin{align}
	\widehat{\mc{N}}(\hat{I}_{a})&:=\frac{1}{\pi}\int_{\hat{I}_{a}}\Im\tr G(E+i\eta_{a})\diff E,
\end{align}
where we recall that $\eta_{a}=N^{-2/3-\delta-\xi}\hat{a}^{-1/3}$. Let $\hat{I}_{a,\epsilon}$ denote the $N^{-2/3-\epsilon}\hat{a}^{-1/3}$-neighbourhood of $\hat{I}_{a}$, and define $\hat{I}_{a,\pm}$ by $\hat{I}_{a,+}:=\hat{I}_{a,\delta+\xi}$ and $\hat{I}_{a}:=(I_{a,-})_{\delta+\xi}$. Then we have
\begin{align}
	\widehat{\mc{N}}(\hat{I}_{a,-})&\leq\mc{N}(\hat{I}_{a})\leq\widehat{\mc{N}}(\hat{I}_{a,+}),
\end{align}
with probability at least $1-N^{-D}$. 

We are now in a position to compare the following smoothed version of $\mc{L}_{m,x,\mbf{k},D,J}$:
\begin{align}
	\mc{S}_{m,x,\mbf{k},D,J}(W)&:=\sum_{i_{1},...,i_{m}}\prod_{a=1}^{m}\psi_{k_{a}+1}(\widehat{\mc{N}}(\hat{I}_{i_{a}}))\prod_{j\in J_{a}}\chi_{x,\alpha}(N\hat{v}_{i_{a},j}^{2}-\ell^{(\beta)}_{N}(D)).
\end{align}  
\begin{lemma}\label{lem:momentmatching2}
Let $m\in\mbb{N}$, $C>0$, $x\in[-C,C]$ and $\max_{a\in[m]}|J_{a}|\leq C$. Then there is a $c>0$ and $\tau>0$ such that
\begin{align}
	\Bigl|\mbb{E}\mc{S}_{m,x,\mbf{k},D,J}(W)-\mbb{E}\mc{S}_{m,x,\mbf{k},D,J}(\wt{W})\Bigr|&\lesssim N^{m-(D+1)\sum_{a}|J_{a}|-c\delta\sum_{a}k_{a}^{2}-\tau}.
\end{align}
\end{lemma}
\begin{proof}
The overall idea of Erd\H{o}s--Xu \cite{erdos_small_2023} is to keep track of the small-probability event in the error estimates that arise from the Lindeberg method. Fix an ordering of upper-triangular indices and let $W^{(\gamma)}$ denote the matrix whose first $\gamma$ upper-triangular entries are those of $W$ and remaining $N(N+1)/2-\gamma$ entries those of $\wt{W}$. Let $W^{(\gamma)}_{0}$ denote the matrix obtained from $W^{(\gamma)}$ by setting the $\gamma$-entry to zero. 

First observe that, for any $p\in\mbb{N}$,
\begin{align}
	\textup{supp}\chi^{(p)}_{x,\alpha}&\subset (x-N^{-\alpha},x+N^{-\alpha}),
\end{align}
and so
\begin{align}
	|\chi^{(p)}_{x,\alpha}|&\lesssim N^{p\alpha}\chi_{x-N^{-\alpha},\alpha}.\label{eq:chi^p}
\end{align}
Moreover, we can choose $\psi_{l}$ such that, for any $p\in\mbb{N}$,
\begin{align}
	|\psi^{(p)}_{l}|&\lesssim \psi_{l}.\label{eq:psi^p}
\end{align} 

Let 
\begin{align}
	F_{x}(W)&:=\mc{S}_{m,x,\mbf{k},D,J}(W)
\end{align}
and observe that, by Lemmas \ref{lem:eigapprox}-\ref{lem:eigapprox2}, \eqref{eq:chi^p} and \eqref{eq:psi^p},
\begin{align}
	|F_{x}^{(p)}(W)|&\leq N^{5p(\alpha+\delta+\xi)}F_{x-\alpha}(W),\label{eq:F^p}
\end{align}
with probability at least $1-N^{-D}$. Indeed, by Lemma \ref{lem:eigapprox} we have
\begin{align}
    \partial_{ab}\hat{N}(\hat{I}_{j})&=\frac{1}{\pi}\Im\tr\bigl(G(\hat{\lambda}_{j}+i\eta_{j}+N^{-2/3-\delta}\hat{j}^{-1/3})-G(\hat{\lambda}_{j}+i\eta_{j}-N^{-2/3-\delta}\hat{j}^{-1/3})\bigr)\partial_{ab}\hat{\lambda}_{j}\\
    &=O\Biggl(\frac{N^{\delta+3\xi}}{N^{2/3}\hat{j}^{1/3}\eta_{j}}\Biggr)\\
    &=O(N^{2\delta+4\xi},)
\end{align}
with probability at least $1-N^{-D}$. Higher order derivatives and derivatives of the other quantities are bounded by similar computations.

By Taylor expansion we have
\begin{align}
	\mbb{E}F_{x}(W^{(\gamma+1)})-\mbb{E}F_{x}(W^{(\gamma)})&=\sum_{p=0}^{4}\mbb{E}F_{x}^{(p)}(W^{(\gamma)}_{0})(W_{\gamma}^{p}-\wt{W}_{\gamma}^{p})\\&+\mbb{E}F_{x}^{(5)}(\theta^{(\gamma)}_{w}(W^{(\gamma)}))(W^{5}_{\gamma}-\wt{W}^{5}_{\gamma}),
\end{align}
for some random variable $w\in[0,1]$. Note that by resolvent expansion, the local law holds for $\theta^{(\gamma)}_{w}(W^{(\gamma)}_{0})$ uniformly in $w\in[0,1]$. By \eqref{eq:F^p} and the moment matching condition, we have
\begin{align}
	|\mbb{E}F_{x}(W^{(\gamma+1)})-\mbb{E}F_{x}(W^{(\gamma)})|&\prec \frac{N^{20(\alpha+\delta+\xi)}t}{N^{2}}\mbb{E}F_{x-N^{-\alpha}}(W^{(\gamma)}_{0})\\
    &+\frac{N^{25(\alpha+\delta+\xi)}}{N^{5/2}}\mbb{E}F_{x-N^{-\alpha}}(\theta^{(\gamma)}_{w}(W^{(\gamma)}_{0})).\label{eq:iteration}
\end{align}
We now proceed iteratively as in Erd\H{o}s--Xu \cite{erdos_small_2023}. In the first iteration, we use the trivial uniform bound $F_{x}(W)\lesssim N^{m}$ to obtain
\begin{align}
	|\mbb{E}F_{x-N^{-\alpha}}(\theta^{(\gamma)}_{w}(W^{(\gamma)}_{0}))-\mbb{E}F_{x-N^{-\alpha}}(\wt{W})|&\lesssim \phi N^{m},
\end{align}
for any $\gamma\in[N(N-1)/2]$ and $w\in[0,1]$, where $\phi=N^{20(\alpha+\delta+\gamma)}t+N^{25(\alpha+\delta+\gamma)-5/2}$. By Proposition \ref{prop:p-gauss-divisible}, we obtain
\begin{align}
	\mbb{E}F_{x-\alpha}(\theta^{(\gamma)}_{w}(W^{(\gamma)}_{0}))&\lesssim \phi N^{m}+N^{m-(D+1)\sum_{a=1}^{m}|J_{a}|-c\delta\sum_{a}k_{a}^{2}}.
\end{align}
In the second iteration, we insert this bound into \eqref{eq:iteration} to obtain
\begin{align}
	|\mbb{E}F_{x}(W)-\mbb{E}F_{x}(\wt{W})|&\lesssim\phi(\phi N^{m}+N^{m-(D+1)\sum_{a=1}^{m}|J_{a}|-c\delta\sum_{a}k_{a}^{2}})
\end{align}
and hence
\begin{align}
	\mbb{E}F_{x}(W)&\lesssim \phi^{2}N^{m}+N^{m-(D+1)\sum_{a=1}^{m}|J_{a}|-c\delta\sum_{a}k_{a}^{2}}.
\end{align}
Continuing in this fashion, after $k$ steps we obtain
\begin{align}
	|\mbb{E}F_{x}(W)-\mbb{E}F_{x}(\wt{W})|&\lesssim \phi^{k}N^{m}+\phi N^{m-(D+1)\sum_{a=1}^{m}|J_{a}|-c\delta\sum_{a}k_{a}^{2}}\\
    &\lesssim \phi N^{m-(D+1)\sum_{a=1}^{m}|J_{a}|-c\delta\sum_{a}k_{a}^{2}},
\end{align}
for sufficiently large $k$ depending on $m,\alpha,\delta,\gamma$.
\end{proof}

\section{Gaussian-Divisible Matrices}
In this section we study Gaussian-divisible matrices, i.e. matrices of the form $W=X+\sqrt{t}Y$ where $X$ is deterministic and $Y$ is from the GOE/GUE. Our main goal is to prove that eigenvector entries can be replaced by i.i.d. Gaussians that are also independent of the eigenvalues. Once this is established, Theorems \ref{thm1}-\ref{thm3} will follow for Gaussian-divisible matrices. 

Let $m\in\mbb{N}$, $I_{a}\subset[m],\,J_{a}\subset[N],\,a=1,...,m$ and set $|J|=\sum_{a=1}^{m}|J_{a}|$. Let $x>0$ and define $g_{m,x}:\mbb{R}^{|J|}\to\mbb{R}$ by 
\begin{align}
    g_{m,x}(\mbf{x})&=\prod_{a=1}^{m}\prod_{j\in J_{a}}\mbf{1}_{(x,\infty)}\Bigl(\sum_{b\in I_{a}}x_{i_{b},j}\Bigr)
\end{align}
Let $f:\mbb{R}^{m}\to\mbb{R}$ and consider the statistic
\begin{align}
	\mc{L}(f,g_{m,x})&:=\sum_{i_{1},...,i_{m}}f(\lambda_{i_{1}},...,\lambda_{i_{m}})g_{m,x}\Bigl(\bigl\{\{N|u_{i_{a},j}|^{2}\}_{j\in J_{a}}\bigr\}_{a\in[m]}\Bigr),\label{eq:L(f,g)}
\end{align}
We define $\mc{L}(f)\equiv\mc{L}(f,1)$. Let $\{Z^{(\beta)}_{i,j}\}_{i\in[N],j\in[N]}$ be a collection of independent standard real ($\beta=1$) or complex ($\beta=2$) Gaussians that are also independent of $W$. The following proposition is the main result of this section.
\begin{proposition}\label{prop:gauss-divisible}
Let $t=N^{-1/3+\rho}$, $m\in\mbb{N}$ and $J_{a}\subset[N],\,a=1,...,m$. Let $f:\mbb{R}^{m}\to\mbb{R}$ be positive and measurable. Assume that
\begin{align}
    \frac{x\log N}{\sqrt{Nt^{3}}}&<N^{-\xi},
\end{align}
for some $\xi>0$. If $\beta=1$ and $m>1$, we additionally assume that there is an $E\in[-2,2]$ and $C>0$ such that
\begin{align}
    \textup{supp}(f)&\subset\Bigl\{\bs\lambda:\max_{i\in[m]}|\lambda_{i}-E|\leq Cd_{E}\Bigr\},
\end{align}
where $d_{E}$ is the mean eigenvalue spacing at $E$. Then we have
\begin{align}
	\mbb{E}\Bigl[\mc{L}(f,g_{m,x})\Bigr]&=\Bigl[1+O\Bigl(\frac{x\log N}{\sqrt{Nt^{3}}}\Bigr)\Bigr]\mbb{E}\Bigl[\mc{L}(f)\Bigr]\mbb{E}\Bigl[g_{m,x}\bigl(\bigl\{\{|Z^{(\beta)}_{i,j}|^{2}\}_{j\in J_{i}}\bigr\}_{i\in[m]}\bigr)\Bigr].
\end{align}
\end{proposition}
The condition on the support of $f$ when $\beta=1$ is purely technical, and the general result should also be true in this case.

Let us give an overview of the proof of Proposition \ref{prop:gauss-divisible} in the simplest case $m=1$. By the method of partial diagonalisation, which we recall in the next subsection, we obtain an integral formula for the expectation
\begin{align}
	\mbb{E}\sum_{i}f(\lambda_{i})g(\{u_{i,j}\}_{j\in J})&=c_{N}\int_{\mbb{R}}f(\lambda)K^{(\beta)}_{1}(\lambda)\mbb{E}^{(\beta)}_{1}\Bigl[g(\{v_{1,j}\}_{j\in J})F_{\beta}(\lambda,\mbf{v}_{1})\Bigr]\diff\lambda,\label{eq:ex}
\end{align}
where the expectation on the right-hand side is with respect to $\mbf{v}_{1}$ distributed according to a certain probability measure $\mu^{(\beta)}_{1}$ on the unit sphere $\textup{S}^{N-1}_{\beta}$, $K^{(\beta)}_{1}(\lambda)$ is the normalisation constant for $\mu^{(\beta)}_{1}$, and the function $F$ is an expectation of a characteristic polynomial:
\begin{align}
	F_{\beta}(\lambda,\mbf{v}_{1})&=\mbb{E}_{Y^{(1)}}\Bigl[|\det(X^{(1)}+\sqrt{t}Y^{(1)}-\lambda)|^{\beta}\Bigr].
\end{align}
Here $X^{(1)}$ is the projection of $X$ onto the space orthogonal to $\mbf{v}_{1}$ and $Y^{(1)}\sim\sqrt{\frac{N}{N-1}}G\beta E$. By the supersymmetry method, we obtain a formula for $F_{\beta}$ in terms of an integral over a finite number of variables. From this representation we deduce that the dependence of $F$ on $\mbf{v}_{1}$ is given by a polynomial in quadratic forms. We can show that such quadratic forms concentrate with respect to $\mu^{(\beta)}_{1}$, and so obtain an approximation of the form
\begin{align}
	F_{\beta}(\lambda,\mbf{v}_{1})&=[1+o(1)]\wt{F}_{\beta}(\lambda).
\end{align}
The dependence on $\mbf{v}_{1}$ has now been removed, so that only $g(\{v_{1,j}\}_{j\in J})$ remains inside the expectation over $\mbf{v}_{1}$. A similar argument to the one giving concentration of quadratic forms also gives Gaussian approximation for the components $v_{1,j}$, and so
\begin{align}
	\mbb{E}^{(\beta)}_{1}[g(\{v_{1,j}\}_{j\in J})]&=[1+o(1)]\mbb{E}\bigl[g(\{Z_{1,j}\}_{j\in J})\bigr].
\end{align} 
Inserting this into \eqref{eq:ex}, we find
\begin{align}
	\mbb{E}\sum_{i}f(\lambda_{i})g(\{u_{i,j}\}_{j\in J})&=[1+o(1)]\int_{\mbb{R}}f(\lambda)K_{\beta}(\lambda)\wt{F}_{\beta}(\lambda)\diff\lambda\cdot\mbb{E}\bigl[g(\{Z_{1,j}\}_{j\in J})\bigr].
\end{align}
We can identify the integral over $\lambda$ with $\mbb{E}\sum_{i}f(\lambda_{i})$, concluding the proof. Note that we do not need to approximate $\wt{F}_{\beta}(\lambda)$ itself, rather we only need to obtain the independence with respect to $\mbf{v}_{1}$.

With Proposition \ref{prop:gauss-divisible} in hand, we can prove Proposition \ref{prop:p-gauss-divisible}.
\begin{proof}[Proof of Proposition \ref{prop:p-gauss-divisible}]
Recall the definition of $\mc{S}_{m,x,\mbf{k},D,J}$:
\begin{align}
	\mc{S}_{m,x,\mbf{k},D,J}(x)&=\sum_{i_{1},...,i_{m}}\prod_{a=1}^{m}\psi_{k_{a}}(\wh{\mc{N}}(\hat{I}_{i_{a}}))\prod_{j\in J_{a}}\chi_{x,\alpha}(N\hat{v}^{2}_{i,j}-\ell^{(\beta)}_{N}(D)),
\end{align}
where $\psi_{k}$ is a smoothed indicator on $[k-1/4,k+1/4]$ and $\chi_{x,\alpha}$ a smoothed indicator on $(x,\infty)$. First we observe that we can replace $\psi_{k}$ with $\mbf{1}_{k}$ and $\chi_{x,\alpha}$ with $\mbf{1}_{(x_{\alpha},\infty)}$, where $x_{\alpha}=x-N^{-\alpha}$. By Lemmas \ref{lem:eigapprox} and \ref{lem:eigapprox2}, for an upper bound we can also replace $\hat{I}_{i}$ with $I_{i}$ (with an additive error of $O(N^{-D})$) and $N\hat{v}^{2}_{i,j}$ with $\sum_{k:\lambda_{k}\in I_{i}}N|u_{k,j}|^{2}$. Finally, we insert the decomposition of unity
\begin{align}
	1&=\prod_{a=1}^{m}\Biggl(\sum_{n_{a}}\zeta\Bigl(\frac{\lambda_{i_{a}}-\gamma_{n_{a}}}{d_{n_{a}}}\Bigr)+r(\lambda_{i_{a}})\Bigr),
\end{align}
where $\zeta$ is a smooth function with compact support, $\textup{supp}(r)\subset(-\infty,-2]\cup[2,\infty)$, $\gamma_{n}$ are the quantiles of the semicircle law and $d_{n}:=N^{-2/3}\hat{n}^{-1/3}$ is the eigenvalue spacing at $\gamma_{n}$.

Let
\begin{align}
	f_{n,k}(\lambda_{1},...,\lambda_{k};\lambda)&=\zeta\Bigl(\frac{\lambda-\gamma_{n}}{d_{n}}\Bigr)\prod_{b=1}^{k}\mbf{1}_{(0,N^{-\delta}d_{n})}(|\lambda_{b}-\lambda|);
\end{align}
then (after possibly adjusting the support of $\zeta$) we have the inequality
\begin{align}
	\zeta\Bigl(\frac{\lambda_{i_{a}}-\gamma_{n_{a}}}{d_{n_{a}}}\Bigr)\mbf{1}_{k_{a}}(\mc{N}(I_{i_{a}}))&\leq\sum_{i_{m+1},...,i_{m+k_{a}}}f_{n_{a},k_{a}}(\lambda_{i_{m+1}},...,\lambda_{i_{m+k}};\lambda_{i_{a}}).
\end{align}
Indeed, this is trivially satisfied if we replace $d_{n_{a}}$ with $d_{i_{a}}$ in the definition of $f_{n,k}$, and we can make such a replacement in the support of $\zeta$.

It therefore suffices to bound
\begin{align}
	\psi_{m,k}&:=\mbb{E}\sum_{i_{1},...,i_{m+k}}\prod_{a=1}^{m}\sum_{n_{a}}f_{n_{a},k_{a}}(\lambda_{i_{m_{a-1}+b}},...,\lambda_{i_{m_{a}}};\lambda_{i_{a}})\nonumber\\
	&\times\prod_{a=1}^{m}\prod_{j\in J_{a}}\mbf{1}_{(x,\infty)}\Bigl(N\sum_{b=m_{a-1}+1}^{m_{a}}|u_{i_{b},j}|^{2}+N|u_{i_{a},j}|^{2}-\ell^{(\beta)}_{N}(D)\Bigr),
\end{align}
where $k=\sum_{a}k_{a}$, $m_{0}=m$, $m_{a}=m+k_{1}+\cdots+k_{a}$, and the sum is over distinct tuples of indices. The sum over the additional indices $i_{m+1},...,i_{m+k}$ represents the extra eigenvalues in the intervals $I_{i_{a}},\,a=1,...,m$, and for each extra eigenvalue in $I_{i_{a}}$ we gain an extra summand in the argument of $\mbf{1}_{(x,\infty)}$. We can now apply Proposition \ref{prop:gauss-divisible} to factorise the eigenvalues and eigenvectors:
\begin{align}
	\psi_{m,k}&\lesssim AB,
\end{align}
where
\begin{align}
	A&=\mbb{E}\sum_{i_{1},...,i_{m+k}}\prod_{a=1}^{m}\sum_{n_{a}}f_{n_{a},k_{a}}(\lambda_{i_{m_{a-1}}+b},...,\lambda_{i_{m_{a}}};\lambda_{i_{a}}),
\end{align} 
and
\begin{align}
	B&=\mbb{E}\prod_{a=1}^{m}\prod_{j\in J_{a}}\mbf{1}_{(x,\infty)}\Bigl(\sum_{b=m_{a-1}+1}^{m_{a}}|Z^{(\beta)}_{b,j}|^{2}+|Z^{(\beta)}_{a,j}|^{2}-\ell^{(\beta)}_{N}(D)\Bigr).
\end{align}
The $Z^{(\beta)}_{i,j}$ are i.i.d. standard Gaussians, for which we have
\begin{align}
	\mbb{E}\mbf{1}_{(y,\infty)}\Bigl(\sum_{i=1}^{n-1}|Z^{(\beta)}_{i}|^{2}+|Z^{(\beta)}_{n}|^{2}\Bigr)&\lesssim y^{\beta n/2-1}e^{-y}.
\end{align}
Substituting $y=x+\ell^{(\beta)}_{N}(D)$ we find
\begin{align}
	B&\lesssim\prod_{a=1}^{m}\prod_{j\in J_{a}}\bigl(\ell^{(\beta)}_{N}(D)+x\bigr)^{k_{a}+1}e^{-\ell^{(\beta)}_{D}}\\
	&\lesssim\Bigl(\frac{x+\log N}{N}\Bigr)^{(D+1)\sum_{a}|J_{a}|}.
\end{align}
For $A$ we use the homogenisation result of Landon--Xian \cite[Theorem 1.2]{landon_edge_2025} (see also Bourgade \cite[Theorem 3.1]{bourgade_extreme_2021}): if $\lambda_{i}(t)$ and $\mu_{i}(t)$ are the eigenvalues of $(1+t)^{-1/2}(X+\sqrt{t}Y)$ and $(1+t)^{-1/2}(\wt{X}+\sqrt{t}Y)$ respectively, where $\wt{X}$ is a GOE/GUE matrix independent of $X$ and $Y$, then for any $D'$ we have
\begin{align}
	\max_{i\in[N]}\mbb{P}\Bigl(\bigl|(\lambda_{i+1}(t)-\lambda_{i}(t))-(\mu_{i+1}(t)-\mu_{i}(t))\bigr|>\frac{1}{Nt^{3}\cdot N^{2/3}\hat{i}^{1/3}}\Bigr)&\leq N^{-D'}.
\end{align}
Applying this again to the matrix $\wt{X}+\sqrt{t}Y=t^{1/2}(Y+t^{-1/2}\wt{X})$, this time considering $Y$ as the initial matrix, we have
\begin{align}
	\max_{i\in[N]}\mbb{P}\Bigl(\bigl|(\lambda_{i+1}(t)-\lambda_{i}(t))-(\nu_{i+1}(t)-\nu_{i}(t))\bigr|>\frac{1}{Nt^{3}\cdot N^{2/3}\hat{i}^{1/3}}\Bigr)&\leq N^{-D'},\label{eq:homogenisation}
\end{align}
where the $\nu_{i}(t)$ are eigenvalues of a GOE/GUE matrix that is independent of $X$ and $Y$. Using this bound we can replace $\lambda_{j}-\lambda_{i}$ with $\mu_{j}-\mu_{i}$ in the arguments of $f_{n,k}$, as long as $Nt^{3}\gg N^{\delta}$ and $|i-j|\ll N^{1-\delta}t^{3}$. By rigidity, we can drop pairs of indices that violate the latter constraint. For example, when $m=1$ we have
\begin{align}
	&\mbb{E}\sum_{i_{1},...,i_{k+1}}\zeta\Bigl(\frac{\lambda_{i_{1}}-\gamma_{n}}{d_{n}}\Bigr)f_{n,k}(\lambda_{i_{2}},...,\lambda_{i_{k+1}};\lambda_{i_{1}})\\
	&=\mbb{E}\sum_{i_{1},...,i_{k+1}}\zeta\Bigl(\frac{\lambda_{i_{1}}-\gamma_{n}}{d_{n}}\Bigr)f_{n,k}(\nu_{i_{2}},...,\nu_{i_{k+1}};\nu_{i_{1}})+O(N^{-D'})\\
	&\leq\mbb{E}\sum_{i_{1}}\zeta\Bigl(\frac{\lambda_{i_{1}}-\gamma_{n}}{d_{n}}\Bigr)\cdot\max_{|\nu-\gamma_{n}|<C\Delta_{n}}\mbb{E}\sum_{i_{2},...,i_{k+1}}f_{n,k}(\nu_{i_{2}},...,\nu_{i_{k+1}};\nu)+O(N^{-D'})\\
	&\leq N^{-c\delta k^{2}}\mbb{E}\sum_{i}\zeta\Bigl(\frac{\lambda_{i}-\gamma_{n}}{d_{n}}\Bigr)+O(N^{-D'}).
\end{align}
In the first equality we use rigidity to restrict $i_{2},...,i_{k+1}$ to values satisfying $|i_{a}-i_{1}|<N^{1-\delta}t^{3}$ and then apply \eqref{eq:homogenisation} to replace $\lambda_{i}$ with $\mu_{i}$ in the arguments of $f_{n,k}$. The factor of $N^{-c\delta k^{2}}$ in last inequality follows from the explicit form of the $k$-point functions of the GOE/GUE, which behave as $\prod_{i<j}|\lambda_{i}-\lambda_{j}|^{\beta}$. Summing the last line over $n$ we find
\begin{align}
	\mbb{E}\sum_{i_{1},...,i_{k+1}}\zeta\Bigl(\frac{\lambda_{i}-\gamma_{n}}{d_{n}}\Bigr)f_{n,k}(\lambda_{i_{2}},...,\lambda_{i_{k+1}};\lambda_{i_{1}})&\leq N^{1-ck^{2}\delta}.
\end{align}
Applying the same arguments for general $m\in\mbb{N}$ we obtain
\begin{align}
	A&\lesssim N^{m-c\delta \sum_{a}k_{a}^{2}}.
\end{align}

Combining the bounds for $A$ and $B$ we have
\begin{align}
	AB&\lesssim N^{m-(D+1)\sum_{a}|J_{a}|-c\delta \sum_{a}k_{a}^{2}}(x+\log N)^{(D+1)\sum_{a}|J_{a}|},
\end{align}
from which Proposition \ref{prop:p-gauss-divisible} follows.
\end{proof}

The rest of this section is devoted to the proof of Proposition \ref{prop:gauss-divisible} and divided into several subsections. In the first subsection we recall the method of partial diagonalisation and prove Proposition \ref{prop:gauss-divisible}, relying on some lemmas that are proved in the second and third subsections.

\subsection{Partial Diagonalisation}\label{sec:partialdiag}
We recall the convention that $\beta=1$ stands for real numbers and $\beta=2$ for complex numbers. Hence $\textup{S}^{n}_{\beta}$ denotes the real or complex $n$-sphere and $\mbb{F}_{\beta}$ denotes $\mbb{R}$ or $\mbb{C}$. Let $H$ be a Hermitian matrix with eigenvalue $\lambda_{1}$. Let $\mbf{v}_{1}$ be the corresponding unit eigenvector and $R_{1}=(\mbf{v}_{1},V_{1})$ the Householder matrix swapping $\mbf{e}_{1}$ and $\mbf{v}_{1}$, where $V_{1}\in\mbb{F}_{\beta}^{N\times(N-1)}$ is the partial isometry projecting onto the space orthogonal to $\mbf{v}_{1}$. Then we have
\begin{align}
    H&=R_{1}\begin{pmatrix}\lambda_{1}&0\\0&H^{(1)}\end{pmatrix}R_{1},
\end{align}
for some Hermitian $H^{(1)}$ of size $N-1$. Let $\lambda_{j}$ be an eigenvalue of $H^{(j)}$ with eigenvector $\mbf{v}_{j}\in \textup{S}_{\beta}^{N-j}$, and $R_{j}=(\mbf{v}_{j},V_{j})$ be the corresponding Householder matrix. Repeating the above transformation $m$ times we obtain
\begin{align}
    H&=U\begin{pmatrix}\Lambda&0\\0&H^{(m)}\end{pmatrix}U^{*},\label{eq:partialDiagonalisation}
\end{align}
where $\Lambda=\textup{diag}(\lambda_{1},...,\lambda_{m})$ and $U$ is a unitary matrix (a product of suitably dilated $R_{j}$). The eigenvectors $\{\mbf{u}_{j}\}$ of $H$ corresponding to $\{\lambda_{j}\}$ are given in terms of $\{\mbf{v}_{j}\}$ by
\begin{align}
    \mbf{u}_{j}&=V_{1}V_{2}\cdots V_{j-1}\mbf{v}_{j}.
\end{align}
Note that, since $V_{j}^{*}V_{j}=1_{N-j}$, for $j<k$ we have
\begin{align*}
    \mbf{u}_{j}^{*}\mbf{u}_{k}&=\mbf{v}_{j}^{*}V_{j-1}^{*}\cdots V_{1}^{*}V_{1}\cdots V_{k-1}\mbf{v}_{k}\\
    &=\mbf{v}_{j}^{*}V_{j}\cdots V_{k-1}\mbf{v}_{k}\\
    &=0,
\end{align*}
so $\mbf{u}_{i}$ are orthonormal as required.

To write down a compact formula for expectation values with respect to a Gaussian-divisible matrices, we first need some definitions. For $j=1,...,m$ and $\lambda_{j}\in\mbb{R}$, define the probability measures $\mu^{(\beta)}_{j}$ on $\textup{S}^{N-j}_{\beta}$ in the following recursive way. Let $X^{(0)}=X$ and $X^{(j)}=V_{j}^{*}X^{(j-1)}V_{j}$, where $\mbf{v}_{j}$ is distributed according to
\begin{align}
    \diff\mu^{(\beta)}_{j}(\mbf{v})&=\frac{1}{K^{(\beta)}_{j}(\lambda_{j},\{\mbf{v}_{i<j}\})}\Bigl(\frac{\beta N}{2\pi t}\Bigr)^{N-j}e^{-\frac{\beta N}{4t}(\mbf{v}^{*}X^{(j-1)}\mbf{v}-\lambda_{j})^{2}-\frac{\beta N}{2t}\|(1-\mbf{v}\mbf{v}^{*})X^{(j-1)}\mbf{v}\|^{2}}\diff S^{(\beta)}_{N-j}(\mbf{v}).
\end{align}
Here $K^{(\beta)}_{j}(\lambda_{j},\{\mbf{v}_{i<j}\})$ is the normalisation and $\diff S^{(\beta)}_{N-j}$ is the unnormalised Haar measure on $\textup{S}^{N-j}_{\beta}$. Note that $\mu^{(\beta)}_{j}$ is defined in terms of $X^{(j-1)}$, and hence depends on $\mbf{v}_{1},...,\mbf{v}_{j-1}$. Define
\begin{align}
	K_{\beta}(\bs\lambda,\{\mbf{v}_{j\leq m}\})&=\prod_{j=1}^{m}K^{(\beta)}_{j}(\lambda_{j},\{\mbf{v}_{i<j}\}),
\end{align}
and
\begin{align}
    F_{\beta}(\bs\lambda,\{\mbf{v}_{j\leq m}\})&=|\Delta(\bs\lambda)|^{\beta}\mbb{E}\Bigl[\prod_{j=1}^{m}\bigl|\det\bigl(X^{(m)}+\sqrt{t}Y^{(m)}-\lambda_{j}\bigr)\bigr|^{\beta}\Bigr],\label{eq:F}
\end{align}
where 
\begin{align}
	\Delta(\mbf{x})&:=\prod_{i<j}(x_{i}-x_{j})
\end{align}
is the Vandermonde determinant and the expectation is with respect to $Y^{(m)}\sim \sqrt{\frac{N}{N-m}}G\beta E(N-m)$. Here and in the following we use the shorthand $\{\mbf{v}_{j\leq m}\}=\{\mbf{v}_{1},...,\mbf{v}_{m}\}$. Note that the dependence of $F_{\beta}$ on $\{\mbf{v}_{j\leq m}\}$ is entirely contained in $X^{(m)}$.

Considering the map in \eqref{eq:partialDiagonalisation} as a change of variables, with $H=X+\sqrt{t}Y$ we have the following formula for expectation values (see \cite[Proposition 4.1]{maltsev_bulk_2024}).
\begin{lemma}\label{lem:partialDiagonalisation}
Let $m\in\mbb{N}$ and $f:\mbb{R}^{m}\times (\textup{S}_{\beta}^{N})^{m}$. Let $X$ be deterministic and $Y\sim G\beta E$ for $\beta\in\{1,2\}$. Then
\begin{align}
    &\mbb{E}\sum_{i_{1},...,i_{m}}f(\lambda_{i_{1}},...,\lambda_{i_{m}},\mbf{u}_{i_{1}},...,\mbf{u}_{i_{m}})\nonumber\\
    &=\Bigl(\frac{\beta N}{4\pi t}\Bigr)^{m/2}\int_{\mbb{R}^{m}}\mbb{E}_{1}\bigl[\cdots\mbb{E}_{m}\bigl[f(\bs\lambda,\{\mbf{u}_{j\leq m}\})F_{\beta}(\bs\lambda,\{\mbf{v}_{j\leq m}\})K_{\beta}(\bs\lambda,\{\mbf{v}_{j\leq m}\})\bigr]\cdots\bigr]\diff\bs\lambda,\label{eq:gaussdivisible}
\end{align}
where $\mbb{E}_{j}$ is the expectation with respect to $\diff\mu^{(\beta)}_{j}$.
\end{lemma}

In the remainder of this subsection we present some preliminary results useful for the analysis of \eqref{eq:gaussdivisible}. We assume throughout that $X$ satisfies the local law, i.e. the estimates in Lemma \ref{lem:locallaw}. We begin with a simple observation: by eigenvalue interlacing, we have
\begin{align}
    |\Tr{G^{(j)}(z)}-\Tr{G(z)}|&\lesssim\frac{j}{N\textup{dist}(z,\sigma(X))},
\end{align}
uniformly in $\{\mbf{v}_{i\leq j}\}$. Thus if $X$ satisfies the averaged local law, then so does $X^{(j)}$. This fact will be used repeatedly in the following.

We define $E^{(j)}_{*}$ by 
\begin{align}
    t\Tr{(G^{(j)}(E^{(j)}_{*}))^{2}}&=1\label{eq:E_*}
\end{align}
and set
\begin{align}
    \lambda^{(j)}_{*}&=E^{(j)}_{*}-\Tr{G^{(j)}(E^{(j)}_{*})}.\label{eq:lambda_*}
\end{align}
This is the spectral edge of the Gaussian-divisible matrix $X^{(j)}+\sqrt{t}Y^{(j)}$, and so we measure the distance from the spectral edge by
\begin{align}
	\kappa^{(j)}_{t}(\lambda)&:=\lambda^{(j)}_{*}-\lambda.\label{eq:kappa}
\end{align}
The asymptotic spectral edge of $X+\sqrt{t}Y$ is $2\sqrt{1+t}$ and so we also define a deterministic equivalent of $\kappa^{(j)}_{t}$ by
\begin{align}
    \kappa_{\lambda}&:=2-\frac{\lambda}{\sqrt{1+t}}.\label{eq:kappa_lambda}
\end{align}

Now we introduce the function
\begin{align}
	\phi^{(j)}_{\lambda}(z)&:=\frac{(z-\lambda)^{2}}{2t}+\Tr{\log(X^{(j)}-z)}.\label{eq:phi}
\end{align}
The functions $\phi^{(j)},\,j=1,...,m$ control the asymptotics of $F_{\beta}$, so we record their relevant properties below.
\begin{lemma}\label{lem:phi}
If $\kappa^{(j)}_{t}(\lambda)>0$, the equation
\begin{align}
	\partial_{z}\phi^{(j)}_{\lambda}(z)&=0\label{eq:phi'}
\end{align}
has two complex conjugate solutions $E^{(j)}_{\lambda}\pm i\eta^{(j)}_{\lambda}$ satisfying
\begin{align}
	E^{(j)}_{\lambda}&=\frac{2+t}{2(1+t)}\lambda+O\Bigl(\frac{1}{N\sqrt{|\kappa_{\lambda}|+t^{2}}}\Bigr),\label{eq:z1}\\
	(\eta^{(j)}_{\lambda})^{2}&=\frac{t^{2}}{1+t}\Bigl(1-\frac{\lambda^{2}}{4(1+t)}\Bigr)+O\Bigl(\frac{\sqrt{|\kappa_{\lambda}|+t^{2}}}{N}\Bigr).\label{eq:z2}
\end{align}
If $\kappa^{(j)}_{t}(\lambda)\leq0$, then \eqref{eq:phi'} has a real solution $E^{(j)}_{\lambda}$ satisfying
\begin{align}
    E^{(j)}_{\lambda}&=\frac{2+t}{2(1+t)}\lambda+\frac{t}{\sqrt{1+t}}\sqrt{\frac{\lambda^{2}}{4(1+t)}-1}+O\Bigl(\frac{1}{N\sqrt{|\kappa_{\lambda}|+t^{2}}}\Bigr).\label{eq:z3}
\end{align}
Moreover, we have
\begin{align}
	\Re\bigl(\phi^{(j)}_{\lambda}(E^{(j)}_{\lambda}\pm i(\eta^{(j)}_{\lambda}+s))-\phi^{(j)}_{\lambda}(E^{(j)}_{\lambda}\pm i\eta^{(j)}_{\lambda})\bigr)&\lesssim-t(|\kappa_{\lambda}|+t^{2})\cdot\frac{\Bigl(\frac{s}{t\sqrt{|\kappa_{\lambda}|+t^{2}}}\Bigr)^{4}}{1+\Bigl(\frac{s}{t\sqrt{|\kappa_{\lambda}|+t^{2}}}\Bigr)^{2}},
\end{align}
uniformly in $s\in[-\eta^{(j)}_{\lambda},\infty)$.
\end{lemma}
These properties follow directly from the local law and are proved in the appendix. In the remainder, we will use the following shorthand:
\begin{align}
    z^{(j)}_{\lambda}&:=E^{(j)}_{\lambda}+i\eta^{(j)}_{\lambda},\\
    G^{(j)}_{\lambda}&:=G^{(j)}(z^{(j)}_{\lambda}),\\
    \phi^{(j)}_{\lambda}&:=\phi^{(j)}_{\lambda}(z^{(j)}_{\lambda}),
\end{align}
and drop the superscript $(j)$ when $j=0$.

We will also need the following variant of $\phi^{(j)}$, which controls the asymptotics of $K_{\beta}$:
\begin{align}
    \psi^{(j)}_{\lambda}(E)&=\frac{(E-\lambda)^{2}-(\eta^{(j)}(E))^{2}}{2t}-\frac{1}{2}\Tr{\log\bigl((X^{(j)}-E)^{2}+(\eta^{(j)}(E))^{2}\bigr)},
\end{align}
where $\eta^{(j)}(E)$ is defined by
\begin{align}
    t\Tr{\bigl((X^{(j)}-E)^{2}+(\eta^{(j)}(E))^{2}\bigr)^{-1}}&=1\label{eq:eta(s)}.
\end{align}
Since the left-hand side is monotonic in $\eta^{2}$, there is a real-valued solution for $(\eta^{(j)}(E))^{2}$. If $E>E^{(j)}_{*}$, then $(\eta^{(j)}(E))^{2}<0$ and hence $\eta^{(j)}(E)$ is imaginary. Note that if $E\leq E^{(j)}_{*}$, then $\psi^{(j)}_{\lambda}(E)=\Re\phi^{(j)}_{\lambda}(E+i\eta^{(j)}(E))$. Also note that $|\eta^{(j)}(E^{(j)}_{\lambda})|=|\eta^{(j)}_{\lambda}|$.
\begin{lemma}\label{lem:phi2}
Let $E\in\mbb{R}$. Then
\begin{align}
	\psi^{(j)}_{\lambda}(E)-\psi^{(j)}_{\lambda}(E^{(j)}_{\lambda})&\gtrsim\frac{(E-E^{(j)}_{\lambda})^{2}}{t}.
\end{align}
\end{lemma}
For ease of notation we define
\begin{align}
    \psi^{(j)}_{\lambda}&:=\psi^{(j)}_{\lambda}(E^{(j)}_{\lambda}),\\
    H^{(j)}(E)&:=\bigl((X^{(j)}-E)^{2}+(\eta^{(j)}(E))^{2}\bigr)^{-1},\label{eq:H}\\
    H^{(j)}_{\lambda}&:=H^{(j)}(E^{(j)}_{\lambda}).
\end{align}

To handle integrals over $\textup{S}^{n}_{\beta}$ we use the following ``duality" formula. The proof proceeds by approximating the Haar measure by a Gaussian (see \cite[Lemma 3.4]{maltsev_bulk_2024}).
\begin{lemma}\label{lem:Sduality}
Let $f:\textup{S}^{n}_{\beta}\to\mbb{C}$ have an extension to $L^{1}(\mbb{F}_{\beta}^{n})$. Then there is an explicit constant $c_{n}>0$ such that
\begin{align}
	\int_{\textup{S}^{n}_{\beta}}f(\mbf{v})\diff S^{(\beta)}_{n}(\mbf{v})&=c_{n}\int_{-\infty}^{\infty}e^{\frac{i\beta p}{2}}\hat{f}_{\beta}(p)\diff p,
\end{align}
where
\begin{align}
	\hat{f}_{\beta}(p)&:=\int_{\mbb{F}^{n}_{\beta}}e^{\frac{i\beta p}{2}\|\mbf{x}\|^{2}}f(\mbf{x})\diff\mbf{x}.
\end{align}
\end{lemma}

The following lemma shows that quadratic forms in $\mbf{v}_{i}$ concentrate, and can be proved in the same way as \cite[Lemma 7.2]{maltsev_bulk_2024} (we give a sketch in the appendix).
\begin{lemma}\label{lem:conc}
Let $A\in\textup{Herm}(N-i)$ and $\mbf{v}_{i}\sim\mu^{(\beta)}_{i}$. Then
\begin{align}
	\Bigl|\mbf{v}_{i}^{*}A\mbf{v}_{i}-t\Tr{AH^{(i-1)}_{\lambda_{i}}}\Bigr|&\prec\frac{t}{\sqrt{N}}\Tr{(AH^{(i-1)}_{\lambda_{i}})^{2}}^{1/2}.\label{eq:conc}
\end{align}
\end{lemma}
Here and in the rest of this section we use the notation $x\prec y$ or $x=O_{\prec}(y)$ to mean
\begin{align}
    x\leq y\log N,
\end{align}
with probability at least $1-e^{-c\log^{2}N}$ with respect to $\{\mbf{v}_{j}\}$.

We also record the following identity relating $G^{(j)}$ and $G^{(j-1)}$.
\begin{lemma}\label{lem:resolventProjection}
Let $V_{j}\in\mbb{F}_{\beta}^{(N-j+1)\times(N-j)}$ be the partial isometry such that $R_{j}=(\mbf{v}_{j},V_{j})$ is a Householder matrix. Then we have
\begin{align}
	V_{j}G^{(j)}(z)V_{j}^{*}&=G^{(j-1)}(z)-\frac{G^{(j-1)}(z)\mbf{v}_{j}\mbf{v}_{j}^{*}G^{(j-1)}(z)}{\mbf{v}_{j}^{*}G^{(j-1)}(z)\mbf{v}_{j}}.
\end{align}
\end{lemma}

We now state three lemmas concerning $F_{\beta}$ and $K_{\beta}$. The first Lemma gives the independence of $F_{2}$ with respect to $\{\mbf{v}_{j}\}$.
\begin{lemma}\label{lem:Ftilde}
There is a function $\wt{F}_{2}(\bs\lambda)$ such that
\begin{align}
	F_{2}(\bs\lambda,\{\mbf{v}_{j\leq m}\})&=\Biggl[1+O_{\prec}\Biggl(\frac{1}{\sqrt{Nt^{3}}}\Biggr)\Biggr]\wt{F}_{2}(\bs\lambda)+O(e^{-c\log^{2}N}),\label{eq:Ftilde2}
\end{align}
uniformly in $\bs{\lambda}$.
\end{lemma}
The second lemma concerns $F_{1}$, where we need to impose additional restrictions for technical reasons.
\begin{lemma}\label{lem:Ftilde1}
Fix $E\in[-2,2]$ and let $\max_{i\in[m]}|\lambda_{i}-E\sqrt{1+t}|<Cd_{E}$, for some $C>0$. Then there is a function $\wt{F}_{1}(\bs\lambda)$ such that
\begin{align}
	F_{1}(\bs\lambda,\{\mbf{v}_{j\leq m}\})&=\Biggl[1+O_{\prec}\Biggl(\frac{1}{\sqrt{Nt^{3}}}\Biggr)\Biggr]\wt{F}_{1}(\bs\lambda)+O(e^{-c\log^{2}N}).\label{eq:Ftilde1}
\end{align}
\end{lemma}
For real matrices, we are only able to prove the desired independence on the local scale, but the result should hold uniformly in $\bs\lambda$.

The third lemma gives the analogous independence for $K^{(\beta)}_{j}(\lambda,\{\mbf{v}_{i<j}\})$.
\begin{lemma}\label{lem:K}
There is a function $\wt{K}^{(\beta)}_{j}$ such that
\begin{align}
    K^{(\beta)}_{j}(\lambda,\{\mbf{v}_{i<j}\})&=\Bigl[1+O_{\prec}\Bigl(\frac{1}{\sqrt{Nt^{3}}}\Bigr)\Bigr]\wt{K}^{(\beta)}_{j}(\lambda).\label{eq:K}
\end{align}
\end{lemma}

The next lemma shows that the entries of $\mbf{u}_{j}$ are well-approximated by iid Gaussians. Let $I\subset[N]$ and define
\begin{align}
    f_{0}(\mbf{v}_{j})&=\prod_{i\in I}\mbf{1}_{(x,\infty)}(N|u_{j,i}|^{2}).
\end{align}
\begin{lemma}\label{lem:gaussianApprox}
Let $\xi^{(\beta)}_{i},\,i\in I$ be iid standard real ($\beta=1$) or  complex ($\beta=2$) Gaussian variables independent of all other random variables. For any $D>0$ we have
\begin{align}
	\mbb{E}^{(\beta)}_{j}[f_{0}(\mbf{v}_{j})]&=\Biggl[1+O_{\prec}\Biggl(\frac{x\log N}{\sqrt{Nt^{3}}}\Biggr)\Biggr]\mbb{E}\Bigl[\prod_{i\in I}\mbf{1}_{(x,\infty)}(|\xi^{(\beta)}_{i}|^{2})\Bigr]+O(N^{-D}).
\end{align}
\end{lemma}

We can now prove Proposition \ref{prop:gauss-divisible} based on the above lemmas.
\begin{proof}[Proof of Proposition \ref{prop:gauss-divisible}]
By Lemma \ref{lem:partialDiagonalisation} and Lemmas \ref{lem:Ftilde} and \ref{lem:Ftilde1}, we have
\begin{align}
    \mbb{E}\mc{L}(f,g_{m,x})&=\Bigl[1+O\Bigl(\frac{\log N}{\sqrt{Nt^{3}}}\Bigr)\Bigr]\Bigl(\frac{\beta N}{4\pi t}\Bigr)^{m/2}\int_{\mbb{R}^{m}}\wt{K}_{\beta}(\bs\lambda)\wt{F}_{\beta}(\bs\lambda)\mbb{E}^{(\beta)}_{1}\Bigl[\cdots\mbb{E}^{(\beta)}_{m}\Bigl[g_{m,x}\Bigr]\cdots\Bigr]\diff\bs\lambda.
\end{align}
The dependence of $g_{m,x}$ on $\mbf{v}_{m}$ is of the form
\begin{align}
    f_{0}(\mbf{v}_{m})&:=\prod_{j\in J}\mbf{1}_{(\Xi,\infty)}\bigl(N|u_{m,j}|^{2}\bigr),
\end{align}
where $\Xi$ is a random variable depending on $\{\mbf{v}_{j\leq m-1}\}$. By Lemma \ref{lem:gaussianApprox}, we have
\begin{align}
    \mbb{E}_{m}[f_{0}(\mbf{v}_{m})]&=\Bigl[1+O\Bigl(\frac{x\log N}{\sqrt{Nt^{3}}}\Bigr)\Bigr]\mbb{E}_{Z}\Bigl[\prod_{j\in J}\mbf{1}_{(\Xi,\infty)}\bigl(|Z^{(\beta)}_{m,j}|^{2}\bigr)].
\end{align}
We now put $\Xi$ back in the argument of a suitable indicator function and repeatedly apply Lemma \ref{lem:gaussianApprox} to replace $|u_{i,j}|^{2}$ with $|Z^{(\beta)}_{i,j}|^{2}$. The outcome of this procedure is
\begin{align}
    \mbb{E}^{(\beta)}_{1}\Bigl[\cdots\mbb{E}^{(\beta)}_{m}\Bigl[g_{m,x}\Bigr]\Bigr]&=\Bigl[1+O\Bigl(\frac{x\log N}{\sqrt{Nt^{3}}}\Bigr)\Bigr]\mbb{E}\Bigl[g_{m,x}\Bigl(\bigl\{\{|Z^{(\beta)}_{i,j}|^{2}\bigr\}_{j\in J_{i}}\}_{i\in [m]}\Bigr].
\end{align}
This can be taken outside the integral over $\bs\lambda$, which reduces to $\mbb{E}\bigl[\mc{L}(f)\bigr]$.
\end{proof}

\subsection{Proof of Lemmas \ref{lem:Ftilde} and \ref{lem:Ftilde1}}\label{sec:charpoly}
The goal of this subsection is to prove Lemmas \ref{lem:Ftilde} and \ref{lem:Ftilde1}, i.e. to show that $F_{\beta}$ is approximately independent of $\{\mbf{v}_{j\leq m}\}$. The approaches for $\beta=1$ and $\beta=2$ are similar, except that in the former case we need an extra step. For this reason we begin with the latter, simpler case.

We start with an identity whose proof is deferred to the appendix (see also Br\'{e}zin--Hikami \cite{brezin_characteristic_2000}).
\begin{lemma}\label{lem:F_2}
We have
\begin{align}
	F_{2}(\bs\lambda,\{\mbf{v}_{j}\})&=\frac{c_{m}N^{2m}}{t^{2m}\Delta^{2}(\bs\lambda)}\int_{\mbb{R}^{2m}}\Delta(\mbf{s})\Biggl(\prod_{j>m}s_{j}\Biggr)e^{N\sum_{j\geq1}\bigl(\phi^{(m)}_{\lambda_{j}}(is_{j})+\phi^{(m)}_{\lambda_{j}}(is_{m+j})\bigr)}\diff\mbf{s},
\end{align}
for some (explicit) constant $c_{m}$.
\end{lemma}

Let
\begin{align}
    \Psi_{j}&=\frac{t\sqrt{|\kappa_{\lambda_{j}}|+t^{2}}}{(Nt(|\kappa_{\lambda_{j}}|+t^{2}))^{1/4}}\log N.
\end{align}
The following lemma shows that we can restrict $\mbf{s}$ to a small neighbourhood of the solutions to $\partial_{z}\phi_{\lambda_{j}}(z)=0$. The proof is a direct application of Lemma \ref{lem:phi} and thus omitted.
\begin{lemma}
Let $q\in[2m]$, and
\begin{align}
	\Omega_{q}&:=\Bigl\{\mu:|is_{j}-E_{\lambda_{j}}-i\eta_{\lambda_{j}}|<\Psi_{j},\,j=1,...,q,\,\,|is_{j}-E_{\lambda_{j}}+i\eta_{\lambda_{j}}|<\Psi_{j},\,j=q+1,...,2m\Bigr\},
\end{align}
where we set $\lambda_{m+j}=\lambda_{j}$. Let $F_{2,\Omega_{q}}$ be the function obtained from $F_{2}$ by restricting the integral over $\mbf{s}$ to $\Omega_{q}$. Then
\begin{align}
	\bigl|F_{2}(\bs\lambda,\{\mbf{v}_{j\leq m}\})-\sum_{q=0}^{2m}F_{2,\Omega_{q}}(\bs\lambda,\{\mbf{v}_{j\leq m}\})\bigr|&\leq e^{-c\log^{2}N},
\end{align}
uniformly in $\bs\lambda$ and $\{\mbf{v}_{j\leq m}\}$.
\end{lemma}

We can now prove Lemma \ref{lem:Ftilde}.
\begin{proof}[Proof of Lemma \ref{lem:Ftilde}]
After localising the integral to $\Omega_{q}$, we can remove the $\mbf{v}_{j}$ dependence using Cramer's rule, the concentration result in Lemma \ref{lem:conc} and eigenvalue interlacing. Indeed, by Cramer's rule we have
\begin{align}
	\frac{\det(X^{(m)}-is)}{\det(X-is)}&=\prod_{j=1}^{m}\mbf{v}_{j}^{*}G^{(j-1)}(is)\mbf{v}_{j}.
\end{align}
By concentration,
\begin{align}
	\mbf{v}_{j}^{*}G^{(j-1)}(is)\mbf{v}_{j}&=t\Tr{H^{(j-1)}_{\lambda_{j}}G^{(j-1)}(is)}+O_{\prec}\Bigl(\frac{t}{\sqrt{N}}\Tr{\bigl(H^{(j-1)}_{\lambda_{j}}G^{(j-1)}(is)\bigr)^{2}}^{1/2}\Bigr).
\end{align}
We can replace $G^{(j-1)}$ with $G$ using interlacing to obtain
\begin{align}
	\mbf{v}_{j}^{*}G^{(j-1)}(is)\mbf{v}_{j}&=\Biggl[1+O_{\prec}\Bigl(\frac{1}{\sqrt{Nt(|\kappa_{\lambda_{j}}|+t^{2})}}\Bigr)\Biggr]t\Tr{H(E_{\lambda_{j}})G(is)}.
\end{align}
Altogether we have
\begin{align}
	\frac{\det(X^{(m)}-is)}{\det(X-is)}&=\prod_{k=1}^{m}\Biggl[1+O_{\prec}\Bigl(\frac{1}{\sqrt{Nt(|\kappa_{\lambda_{j}}|+t^{2})}}\Bigr)\Biggr]t\Tr{H(E_{\lambda_{k}})G(is)}.
\end{align}
Inserting this into $F_{2}$ we find
\begin{align}
	F_{2}(\bs\lambda,\{\mbf{v}_{j\leq m}\})&=c_{m}\Biggl(\frac{N}{t}\Biggr)^{2m}\frac{1}{\Delta^{2}(\lambda)}\int_{\mbb{R}^{2m}}\Delta(\mbf{s})\Biggl(\prod_{j>m}s_{j}\Biggr)e^{N\sum_{j\geq1}\Bigl(\phi_{\lambda_{j}}(is_{j})+\phi_{\lambda_{j}}(is_{m+j})\Bigr)}\nonumber\\
	&\times\Biggl[1+O_{\prec}\Biggl(\frac{1}{\sqrt{Nt^{3}}}\Biggr)\Biggr]\prod_{j=1}^{2m}\prod_{k=1}^{m}t\Tr{H(E_{\lambda_{k}})G(is_{j})}\diff\mbf{s}.
\end{align}
Since the $o(1)$ term inside the integral is symmetric in $s_{j}$, any cancellations in the integral due to $\Delta(\mbf{s})$ occur for both the leading term and the error term, and hence we can take the error outside the integral. Defining
\begin{align}
	\wt{F}_{2}(\bs\lambda)&=c_{m}\Biggl(\frac{N}{t}\Biggr)^{2m}\frac{1}{\Delta^{2}(\bs\lambda)}\int_{\mbb{R}^{2m}}\Delta(\mbf{s})\Biggl(\prod_{j>m}s_{j}\Biggr)e^{N\sum_{j\geq1}\Bigl(\phi_{\lambda_{j}}(s_{j})+\phi_{\lambda_{j}}(s_{m+j})\Bigr)}\prod_{j=1}^{2m}f(\bs\lambda,s_{j})\diff\mbf{s},
\end{align}
with
\begin{align}
	f(\bs\lambda,s)&=\prod_{k=1}^{m}t\Tr{H(E_{\lambda_{k}})G(is)},
\end{align}
we obtain \eqref{eq:Ftilde2}.
\end{proof}

We now come to real matrices, i.e. $\beta=1$. First we use the method of spin variables (see \cite[Lemma 5.2]{osman_bulk_2025}) to remove absolute values from all but one characteristic polynomial, after which we argue in a similar way to the proof of Lemma \ref{lem:Ftilde}. Define
\begin{align}
	s(\lambda)&:=\lim_{\epsilon\to 0}\frac{\det(X-\lambda-\epsilon)}{|\det(X-\lambda-\epsilon)|},
\end{align}
which is the left-continuous version of  $\textup{sign}(\det(X-\lambda))$. We have the following representation.
\begin{lemma}\label{lem:spin}
Let $X$ be a random matrix with distinct eigenvalues almost surely, and let $\mu_{1}<\cdots<\mu_{m}$. Then there are constants $c_{k}$ such that
\begin{align}
	\mbb{E}\prod_{j=1}^{m}|\det(X-\mu_{j})|&=\sum_{k=0}^{m-1}c_{k}\mbb{E}\sum_{i_{1},...,i_{k}}\mbf{1}_{\mc{I}_{k}}(\lambda_{i_{1}},...,\lambda_{i_{k}})s(\lambda_{i_{1}})\cdots s(\lambda_{i_{k}})\prod_{j}\det(X-\mu_{j})\nonumber\\
	&+c_{m}\mbb{E}\sum_{i_{1},...,i_{m-1}}\mbf{1}_{\mc{I}_{m-1}}(\lambda_{i_{1}},...,\lambda_{i_{m-1}})s(\lambda_{i_{1}})\cdots s(\lambda_{i_{m-1}})\nonumber\\
	&\times\Biggl(\prod_{j=1}^{m-1}\det(X-\mu_{j})\Biggr)|\det(X-\mu_{m})|,\label{eq:spin}
\end{align}
where
\begin{align}
	\mc{I}_{k}&:=\{(\lambda_{1},...,\lambda_{k}):\mu_{1}<\lambda_{1}<\cdots<\mu_{k}<\lambda_{k}<\mu_{k+1}\}.
\end{align}
\end{lemma}
\begin{proof}
Since $s(\lambda)$ alternates between $+1$ and $-1$ as $\lambda$ passes through eigenvalues, we have
\begin{align}
	s(\lambda)&=2\sum_{\lambda_{n}<\lambda}s(\lambda_{n})-1\\
	&=3s(\mu)+2\sum_{\mu<\lambda_{n}<\lambda}s(\lambda_{n}).
\end{align}
We write $|\det(X-\lambda)|=s(\lambda)\det(X-\lambda)$ and use the above identity to conclude. For example,
\begin{align}
	|\det(X-\mu_{1})||\det(X-\mu_{2})|&=s(\mu_{2})|\det(X-\mu_{1})|\det(X-\mu_{2})\\
	&=\Biggl(3s(\mu_{1})+2\sum_{\mu_{1}<\lambda_{i}<\mu_{2}}s(\lambda_{i})\Biggr)|\det(X-\mu_{1})|\det(X-\mu_{2})\\
	&=3\det(X-\mu_{1})\det(X-\mu_{2})\nonumber\\
	&+2\sum_{i}\mbf{1}(\{\mu_{1}<\lambda_{i}<\mu_{2}\})s(\lambda_{i})|\det(X-\mu_{1})|\det(X-\mu_{2}).
\end{align}
\end{proof}
The main difference between this representation and the one in \cite[Lemma 5.2]{osman_bulk_2025} is the fact that all new eigenvalues in the sums are contained in $[\mu_{1},\mu_{m}]$, whereas in the representation in \cite[Lemma 5.2]{osman_bulk_2025} one eigenvalue is contained in $(\mu_{m},\infty)$. The cost of having all eigenvalues in a compact set is that we have to keep one absolute value in the last term in \eqref{eq:spin}. For this term we use the identity
\begin{align}
	|\det(X-\mu)|&=\lim_{\epsilon\to 0}\frac{\det^{2}(X-\mu)}{\det^{1/2}((X-\mu)^{2}+\epsilon)}.
\end{align}

Each term in \eqref{eq:spin} is the expectation value of a sum over tuples of distinct eigenvalues and hence can be evaluated by partial diagonalisation. The ``spin" variables $s(\lambda_{i})$ will remove the absolute value from the Jacobian $|\det(X-\lambda_{i})|$. We are left with the functions
\begin{align}
	F_{1,1}(\bs\lambda,\{\mbf{v}_{j\leq n}\})&:=\wt{\Delta}(\bs\lambda)\mbb{E}\Bigl[\prod_{j=1}^{n}\det(X^{(n)}+\sqrt{t}Y^{(n)}-\lambda_{j})\Bigr],
\end{align}
and
\begin{align}
	F_{1,2}(\bs\lambda,\{\mbf{v}_{j\leq n}\})&:=\wt{\Delta}(\bs\lambda)\mbb{E}\Bigl[|\det(X^{(n)}+\sqrt{t}Y^{(n)}-\lambda_{1})|\prod_{j=2}^{n}\det(X^{(n)}+\sqrt{t}Y^{(n)}-\lambda_{j})\Bigr].
\end{align}
where $n\geq m$ and
\begin{align}
	\wt{\Delta}(\bs\lambda)&=\prod_{j<k}^{m}|\lambda_{j}-\lambda_{k}|\prod_{m<j<k}(\lambda_{j}-\lambda_{k})\prod_{j=1}^{m}\prod_{k=m+1}^{n}(\lambda_{k}-\lambda_{j}).
\end{align}
By integrating these functions over the excess variables $\lambda_{m+1},...,\lambda_{n}$ and $\mbf{v}_{m+1},...,\mbf{v}_{m+n}$, we can recover $F_{1}(\bs\lambda,\{\mbf{v}_{j\leq m}\})$.

Similarly to Lemma \ref{lem:F_2}, we have an integral formula for $F_{1,1}$ (see also Br\'{e}zin--Hikami \cite{brezin_characteristic_2001}).
\begin{lemma}\label{lem:F_1,1}
We have
\begin{align}
	F_{1,1}(\bs\lambda)&=\wt{\Delta}(\bs\lambda)\int_{\mbb{R}^{n}}\Delta^{4}(\mbf{s})e^{N\sum_{j}\phi^{(n)}_{\lambda_{1}}(is_{j})}I_{n}(\mbf{s})\diff\mbf{s},
\end{align}
where
\begin{align}
	I_{n}(\mbf{s})&=\int_{\textup{Sp}(n)}e^{\frac{iN}{t}\tr UJ(\mbf{s})U^{*}J(\bs\lambda-\lambda_{1})}\diff U.
\end{align}
Here $\textup{Sp}(n)$ denotes the compact symplectic group of order $n$ and
\begin{align}
	J(\mbf{x})&=\begin{pmatrix}0&\textup{diag}(\mbf{x})\\-\textup{diag}(\mbf{x})&0\end{pmatrix}.
\end{align}
\end{lemma}

To write down a similar identity for $F_{1,2}$, we first introduce the probability measure
\begin{align}
	\diff\nu_{r}(\mbf{u};r)&=\frac{1}{K_{n}(r,\lambda_{1},\{\mbf{v}_{j\leq n}\})}\Biggl(\frac{Nr}{2\pi t(1+r)}\Biggr)^{(N-n)/2-1}\nonumber\\
    &\times e^{\frac{Nr}{2t(1+r)}\Bigl(\mbf{u}^{T}(X^{(n)}-\lambda_{1})^{2}\mbf{u}-\frac{r}{1+2r}\bigl(\mbf{u}^{T}(X^{(n)}-\lambda_{1})\mbf{u}\bigr)^{2}\Bigr)}\diff S^{(1)}_{N-n}(\mbf{u}),
\end{align}
where $K_{n}(r,\lambda_{1},\{\mbf{v}_{j\leq m}\})$ is the normalisation. This is a one-parameter extension of $\mu^{(1)}_{n}$, which is recovered in the limit $r\to\infty$. Let $X^{(n+1)}$ denote the projection of $X^{(n)}$ onto the space orthogonal to $\mbf{u}$ and $\mbb{E}_{r}$ denote the expectation with respect to $\mbf{u}\sim\nu_{r}$. We also define
\begin{align}
	p(x,U,\mbf{s},r)&:=\mbb{E}_{r}\Biggl[\pf\Bigl(\frac{1}{1+r}\bigl((1_{2n}\otimes\mbf{u}^{T})(J\otimes X^{(n)}-iJ(\mbf{s})\otimes 1_{N-n})^{-1}(1_{2n}\otimes\mbf{u})\bigr)^{-1}+B(\mbf{s},r)\Bigr)\Biggr],
\end{align}
where
\begin{align}
	B(\mbf{s},r)&=\frac{ir}{1+r}J(\mbf{s})+\frac{r^{2}}{(1+r)(1+2r)}\bigl(\mbf{u}^{T}(X^{(n)}-\lambda_{1})\mbf{u}+x\bigr)J-rU^{*}J(\wt{\bs\lambda}-\lambda_{1})U,
\end{align}
and $\wt{\bs\lambda}=(\lambda_{1},\lambda_{1},\lambda_{2},\lambda_{3},...,\lambda_{n})$. Here we have used the shorthand $J\equiv J(0)$ for the symplectic form.
\begin{lemma}\label{lem:F_1,2}
We have
\begin{align}
	F_{1,2}(\bs\lambda,\{\mbf{v}\})&=\wt{\Delta}(\bs\lambda)\int_{0}^{\infty}\frac{K_{n}(r,\lambda_{1},\{\mbf{v}_{j\leq n}\})}{(1+2r)^{1/2}(1+r)^{n+1/2}}\int_{\mbb{R}^{n+1}}\Delta^{4}(\mbf{s})e^{N\sum_{j}\phi^{(n+1)}_{\lambda_{1}}(is_{j})}f(\mbf{s},r)\diff\mbf{s}\diff r,\label{eq:F_1,2}
\end{align}
where
\begin{align}
	f(\mbf{s},r)&=\int_{-\infty}^{\infty}e^{-\frac{Nr^{2}x^{2}}{4t(1+r)^{2}(1+2r)}}\int_{\textup{Sp}(n+1)}e^{\frac{iN}{t}\tr UJ(\mbf{s})U^{*}J(\wt{\bs\lambda}-\lambda_{1})}p(x,U,\mbf{s},r)\diff U\diff x,
\end{align}
\end{lemma}
The measure $\nu_{r}$ can be treated in a similar way to $\mu^{(n)}_{n}$. In order to do so we define $\eta^{(n)}(E,r)$ by 
\begin{align}
    t\Tr{\bigl((X^{(n)}-E)^{2}+(\eta^{(n)}(E,r))^{2}\bigr)^{-1}}&=\frac{r}{1+r},\label{eq:eta(E,r)}
\end{align}
and $E^{(n)}_{\lambda,r}$ by
\begin{align}
    E^{(n)}_{\lambda,r}&=\lambda+t\Tr{(X^{(n)}-E^{(n)}_{\lambda,r})\bigl((X^{(n)}-E^{(n)}_{\lambda,r})^{2}+(\eta^{(n)}(E^{(n)}_{\lambda,r}))^{2}\bigr)^{-1}}.
\end{align}
We also define
\begin{align}
    H^{(n)}_{\lambda}(r)&:=\bigl((X^{(n)}-E^{(n)}_{\lambda,r})^{2}+(\eta^{(n)}(E^{(n)}_{\lambda,r},r))^{2}\bigr)^{-1},
\end{align}
Firstly, we have the analogue of Lemma \ref{lem:K}.
\begin{lemma}\label{lem:K(r)}
There is a function $\wt{K}_{n}(r,\lambda)$ such that
\begin{align}
    K_{n}(r,\lambda,\{\mbf{v}_{j\leq n}\})&=\Bigl[1+O_{\prec}\Bigl(\frac{\log N}{\sqrt{Nt^{3}}}\Bigr)\Bigr]\wt{K}_{n}(r,\lambda).
\end{align}
\end{lemma}

Secondly, we have the analogue of Lemma \ref{lem:conc}. The proof is exactly the same and hence omitted.
\begin{lemma}\label{lem:nuConc}
Let $r>0$ and $\mbf{u}\sim\nu_{r}$. Then for any $A\in\textup{Sym}(N-n)$ we have
\begin{align}
	\Bigl|\mbf{u}^{T}A\mbf{u}-t\Tr{AH^{(n)}_{\lambda_{1}}(r)}\Bigr|&\prec\frac{t}{\sqrt{N}}\Tr{(AH^{(n)}_{\lambda_{1}}(r))^{2}}^{1/2}.
\end{align}
\end{lemma}

The integral over $\mbf{s}$ is very similar to that in $F_{1,1}$, the only difference being the function $p(x,U,\mbf{s},r)$. Given the above properties of $\nu_{r}$, the rest of the analysis is very similar to that of $F_{2}$. Ultimately we obtain
\begin{lemma}\label{lem:Ftilde_1j}
There are functions $\wt{F}_{1,j}:\mbb{R}^{n}\to\mbb{C}$ such that
\begin{align}
	F_{1,j}(\bs\lambda,\{\mbf{v}_{j\leq n}\})&=\Biggl[1+O_{\prec}\Biggl(\frac{1}{\sqrt{Nt^{3}}}\Biggr)\Biggr]\wt{F}_{1,j}(\bs\lambda)+O(e^{-c\log^{2}N}).
\end{align}
\end{lemma}
\begin{proof}[Proof sketch]
The integral over $\mbf{s}$ is controlled by $e^{N\sum_{j}\phi^{(n)}_{\lambda_{1}}(is_{j})}$ and treated by the method of steepest descent using Lemma \ref{lem:phi}. Note that $I_{n}(\mbf{s})$ does not contribute to the saddle-point due to the assumption that
\begin{align}
    \max_{i}|\lambda_{i}-E\sqrt{1+t}|<Cd_{E}.
\end{align}
We then replace $\phi^{(n)}_{\lambda}$ by $\phi_{\lambda}$ as in the proof of Lemma \ref{lem:Ftilde}. The quadratic forms in $\mbf{u}$ that appear in the function $p(x,U,\mbf{s},r)$ can be approximated with traces of resolvents of $X^{(n)}$ by Lemma \ref{lem:nuConc}. These traces are themselves approximated by traces of resolvents of $X$ using interlacing. 
\end{proof}

We can now prove Lemma \ref{lem:Ftilde1}
\begin{proof}
For simplicity we focus on the case $m=2$; the general case is exactly the same. By Lemma \ref{lem:spin}, we have
\begin{align}
	\mbb{E}\prod_{j=1}^{2}|\det(X-\lambda_{1})||\det(X-\lambda_{2})|&=c_{1}A_{1}+c_{2}A_{2},
\end{align}
where
\begin{align}
	A_{1}&=\mbb{E}\prod_{j=1}^{2}\det(X-\lambda_{1})\det(X-\lambda_{2}),
\end{align}
and
\begin{align}
	A_{2}&=\mbb{E}\sum_{i}\mbf{1}_{(\lambda_{1},\lambda_{2})}(\lambda_{i})s(\lambda_{i})\det(X-\lambda_{1})|\det(X-\lambda_{2})|.
\end{align}
By Lemma \ref{lem:Ftilde_1j} and Lemma \ref{lem:K}, we have
\begin{align}
	A_{1}&=\Biggl[1+O\Biggl(\frac{1}{\sqrt{Nt^{3}}}\Biggr)\Biggr]\wt{F}_{1,1}(\lambda_{1},\lambda_{2}),\\
	A_{2}&=\int_{\lambda_{1}}^{\lambda_{2}}\Biggl[1+O\Biggl(\frac{1}{\sqrt{Nt^{3}}}\Biggr)\Biggr]\wt{K}^{(1)}_{3}(\lambda)\wt{F}_{1,2}(\lambda_{2},\lambda_{1},\lambda)\diff\lambda.
\end{align}
To take the error outside the integral over $\lambda$ in $A_{2}$, we need to ensure that the leading term does not oscillate rapidly in $\lambda$, but this is clear from the steepest descent analysis in Lemma \ref{lem:Ftilde_1j}, in which only $\lambda_{2}$ plays a role. It is here that we need the assumption $|\lambda_{i}-\lambda_{j}|<Cd_{E}$, since otherwise the term $e^{\frac{iN}{t}\tr UJ(\mbf{s})U^{*}J(\bs\lambda-\lambda_{1})}$ would contribute to the saddle point and we would not be able to conclude that it does not oscillate rapidly in any of the $\lambda_{j}$. In general, the dominant contribution is given by a sum of terms of the form
\begin{align}
	\Re\Bigl(e^{N\sum_{j}\phi^{\pm}_{\lambda_{j}}}\Bigr).
\end{align}
If the $\lambda_{j}$ are far apart, then the imaginary parts of $\phi^{\pm}_{\lambda_{j}}=\phi_{\lambda_{j}}(E_{j}\pm i\eta_{j})$ do not cancel and the end result is oscillatory.
\end{proof}

\subsection{Expectation with respect to $\mu^{(\beta)}_{j}$}\label{sec:E_j}
In this subsection we study expectation values of indicator functions with respect to $\mu^{(\beta)}_{j}$. Recall that
\begin{align}
	\mbb{E}^{(\beta)}_{j}\bigl[f(\mbf{v}_{j})\bigr]&=\frac{1}{K^{(\beta)}_{j}(\lambda_{j},\{\mbf{v}_{i<j}\})}\Biggl(\frac{\beta N}{2\pi t}\Biggr)^{N-j}\int_{\textup{S}^{N-j}_{\beta}}f(\mbf{v}_{j})\nonumber\\
	&\times e^{-\frac{\beta N}{2t}\mbf{v}_{j}^{*}\bigl(X^{(j-1)}-\lambda_{j}\bigr)^{2}\mbf{v}_{j}+\frac{\beta N}{4t}\bigl(\mbf{v}_{j}^{*}(X^{(j-1)}-\lambda_{j})\mbf{v}_{j}\bigr)^{2}}\diff S^{(\beta)}_{N-j}(\mbf{v}_{j}).\label{eq:E_j}
\end{align}
We first obtain a general identity for expectation values, which follows directly from Lemma \ref{lem:Sduality}.
\begin{lemma}\label{lem:E_j}
We have
\begin{align}
	\mbb{E}^{(\beta)}_{j}\bigl[f(\mbf{v}_{j})\bigr]&=\frac{1}{K^{(\beta)}_{j}(\lambda_{j},\{\mbf{v}_{i<j}\}}\sqrt{\frac{2\beta N}{\pi t}}\int_{-\infty}^{\infty}e^{-\beta N\psi^{(j-1)}_{\lambda_{j}}(E)}\nonumber\\
	&\times\int_{-\infty}^{\infty}e^{\frac{i\beta Np}{2t}}\det^{-\beta/2}\bigl(1+ipH^{(j-1)}(E)\bigr)h_{f}(p,E)\diff p\diff E,
\end{align}
where $H^{(j)}(E)$ is defined in \eqref{eq:H} and
\begin{align}
	h_{f}(p)&=\Biggl(\frac{\beta N}{2\pi t}\Biggr)^{N-j}\det\bigl((X^{(j-1)}-E)^{2}+(\eta^{(j-1)}(E))^{2}\bigr)\nonumber\\
    &\times\int_{\mbb{F}_{\beta}^{N-j}}e^{-\frac{\beta N}{2}\mbf{x}^{*}\bigl((X^{(j-1)}-E)^{2}+(\eta^{(j-1)}(E))^{2}+ip\bigr)\mbf{x}}f(\mbf{x})\diff\mbf{x}.
\end{align}
\end{lemma}

By setting $f=1$ we obtain a formula for $K^{(\beta)}_{j}(\lambda,\{\mbf{v}_{i<j}\})$, which we use to prove Lemma \ref{lem:K}.
\begin{proof}[Proof of Lemma \ref{lem:K}]
Using the formula in Lemma \ref{lem:E_j} with $f=1$, we have
\begin{align}
    K^{(\beta)}_{j}(\lambda,\{\mbf{v}_{i<j}\})&=\sqrt{\frac{2\beta N}{\pi t}}\int_{-\infty}^{\infty}e^{-\beta N\psi^{(j-1)}_{\lambda_{j}}(E)}g(E)\diff E,
\end{align}
where
\begin{align}
    g(E)&=\int_{-\infty}^{\infty}e^{\frac{i\beta Np}{2t}}\det^{-\beta/2}\bigl(1+ipH^{(j-1)}(E)\bigr)\diff p.
\end{align}

Using the bound
\begin{align}
    |\det^{-1}(1+ipA)|&\leq \exp\Bigl\{-\frac{p^{2}\tr A^{2}}{1+p^{2}\|A\|^{2}}\Bigr\},
\end{align}
we restrict $p$ to the region $|p|<\frac{\log N}{\sqrt{N\Tr{\bigl(H^{(j-1)}(E)\bigr)^{2}}}}$, in which we make a Taylor expansion using the definition of $\eta^{(j-1)}(E)$ to obtain
\begin{align}
    e^{\frac{i\beta N p}{2t}}\det^{-\beta/2}\bigl(1+ipH^{(j-1)}(E)\bigr)&=\Bigl[1+O\Bigl(\frac{\log N}{\sqrt{Nt^{3}}}\Bigr)\Bigr]\exp\Bigl\{-\frac{\beta Np^{2}\Tr{(H^{(j-1)}(E))^{2}}}{4}\Bigr\}.
\end{align}
Thus we find
\begin{align}
    g(E)&=\Bigl[1+O\Bigl(\frac{\log N}{\sqrt{Nt^{3}}}\Bigr)\Bigr]\sqrt{\frac{4\pi}{\beta N\Tr{(H^{(j-1)}(E))^{2}}}}\\
    &=\Bigl[1+O\Bigl(\frac{\log N}{\sqrt{Nt^{3}}}\Bigr)\Bigr]\sqrt{\frac{4\pi}{\beta N\Tr{H^{2}(E)}}},
\end{align}
where in the second line we used interlacing to replace $H^{(j-1)}$ with $H$.

Now we replace $\eta^{(j-1)}(E)$ with $\eta(E)$ using the estimate
\begin{align}
    &e^{\frac{N}{t}\bigl((\eta^{(i)}(E))^{2}-(\eta^{(i-1)}(E))^{2}\bigr)}\frac{\det\bigl((X^{(i)}-E)^{2}+(\eta^{(i-1)}(E))^{2}\bigr)}{\det\bigl((X^{(i)}-E)^{2}+(\eta^{(i)}(E))^{2}\bigr)}\nonumber\\&=e^{\frac{N}{t}\bigl((\eta^{(i)}(E))^{2}-(\eta^{(i-1)}(E))^{2}\bigr)}\det\bigl(1-\bigl((\eta^{(i)}(E))^{2}-(\eta^{(i-1)}(E))^{2}\bigr)H^{(i)}(E)\bigr)\\
    &=1+O\Bigl(\frac{1}{Nt^{3}}\Bigr),
\end{align}
where the last line follows by the approximation of $\eta(E)$ in \eqref{eq:eta(E)approx}. We also use Cramer's rule to replace $X^{(i)}$ by $X^{(i-1)}$:
\begin{align}
    \frac{\det\bigl((X^{(i-1)}-E)^{2}+(\eta^{(i-1)}(E))^{2}\bigr)}{\det\bigl((X^{(i)}-E)^{2}+(\eta^{(i-1)}(E))^{2}\bigr)}&=\frac{1}{\mbf{v}_{i}^{*}G^{(i-1)}(E+i\eta^{(i-1)}(E))\mbf{v}_{i}\mbf{v}_{i}^{*}G^{(i-1)}(E-i\eta^{(i-1)}(E))\mbf{v}_{i}}.
\end{align}
Using the concentration of quadratic forms in Lemma \ref{lem:conc}, interlacing and the approximation of $\eta(E)$ in \eqref{eq:eta(E)approx}, we have
\begin{align}
    \mbf{v}_{i}^{*}G^{(i-1)}(E\pm i\eta^{(i-1)}(E))\mbf{v}_{i}&=\Bigl[1+O\Bigl(\frac{1}{Nt^{3}}\Bigr)\Bigr]t\Tr{H^{(i-1)}_{\lambda_{i}}G^{(i-1)}(E\pm i\eta^{(i-1)}(E))}\\
    &=\Bigl[1+O_{\prec}\Bigl(\frac{1}{Nt^{3}}\Bigr)\Bigr]t\Tr{H(E_{\lambda_{i}})G(E\pm i\eta(E))}.
\end{align}
Combining these estimates, we can replace $\psi^{(j-1)}_{\lambda}$ with $\psi_{\lambda}$:
\begin{align}
    e^{\beta N\bigl(\psi_{\lambda}(E)-\psi^{(j-1)}_{\lambda}(E)\bigr)}&=\Bigl[1+O_{\prec}\Bigl(\frac{1}{Nt^{3}}\Bigr)\Bigr]\chi(E),
\end{align}
where
\begin{align}
    \chi(E)&=\prod_{i=1}^{j-1}\frac{1}{t^{2}\Tr{H(E_{\lambda_{i}})G(E+i\eta(E))}\Tr{H(E_{\lambda_{i}})G(E-i\eta(E))}}.
\end{align}
Defining
\begin{align}
    \wt{K}^{(\beta)}_{j}(\lambda)&=\sqrt{\frac{8}{t}}\int e^{-\beta N\psi_{\lambda_{j}}(E)}\frac{\chi(E)}{\sqrt{\Tr{H^{2}(E)}}}\diff E,
\end{align}
we obtain \eqref{eq:K}.
\end{proof}

We are interested in expectation values of 
\begin{align}
	f_{\epsilon}(\mbf{v}_{j})&=\prod_{i\in I}e^{-\epsilon(N|u_{j,i}|^{2}-x)}\mbf{1}_{(x,\infty)}\bigl(N|u_{j,i}|^{2}\bigr),
\end{align}
where $\mbf{u}_{j}\equiv\mbf{u}_{j}(\mbf{v}_{1},...,\mbf{v}_{j})$ is an eigenvector of the original Gaussian-divisible matrix. By the delocalisation bound $N\|\mbf{u}_{j}\|_{\infty}^{2}\prec1$, if $\epsilon=N^{-\delta}$ for any $\delta>0$ then
\begin{align}
	f_{\epsilon}(\mbf{v}_{j})&=f_{0}(\mbf{v}_{j})
\end{align}
with overwhelming probability. We recall that $\mbf{u}_{j}=V_{1}\cdots V_{j-1}\mbf{v}_{j}$, so that
\begin{align}
    |u_{j,i}|^{2}&=\mbf{v}_{j}^{*}\mbf{e}^{(j-1)}_{i}\mbf{e}^{(j-1)*}_{i}\mbf{v}_{j},
\end{align}
where we have defined the projected coordinate vectors
\begin{align}
    \mbf{e}^{(j)}_{i}&:=V_{j}^{*}\cdots V_{1}^{*}\mbf{e}_{i}.
\end{align}

To compute integrals of indicator functions, we use the following.
\begin{lemma}\label{lem:halfSpace}
Let $a_{j}\in\mbb{R}$ and $A\in\mbb{C}^{n\times n}$ such that $\Re A:=\frac{1}{2}(A+A^{*})>0$. Then for any $\epsilon>0$ and $J\subset[n]$ we have the identity
\begin{align}
    &\det^{\beta/2}\Biggl(\frac{\beta A}{2\pi}\Biggr)\int_{\mbb{F}_{\beta}^{n}}e^{-\frac{\beta}{2}\mbf{x}^{*}A\mbf{x}}\prod_{j\in J}e^{-\epsilon(|x_{j}|^{2}-a_{j})}1_{(a_{j},\infty)}(|x_{j}|^{2})\diff\mbf{x}\nonumber\\
    &=\int_{\mbb{R}^{|J|}}\det^{-\beta/2}(1-iD(\mbf{k})(A^{-1})_{J})\prod_{j=1}^{|J|}\frac{e^{-ik_{j}a_{j}}}{k_{j}-i\epsilon}\frac{\diff k_{j}}{2\pi},\label{eq:halfSpace}
\end{align}
where $(A^{-1})_{J}$ is the principal submatrix of $A^{-1}$ labelled by indices in $J$.
\end{lemma}

For an index set $I\subset[N]$, define
\begin{align}
    I^{(j)}&:=\{\mbf{e}^{(j)}_{i}:i\in I\}.
\end{align}
Using this lemma we obtain
\begin{align}
    \mbb{E}^{(\beta)}_{j}[f_{\epsilon}(\mbf{v}_{j})]&=\frac{1}{K^{(\beta)}_{j}(\lambda_{j},\{\mbf{v}_{i<j}\})}\sqrt{\frac{2\beta N}{\pi t}}\int_{\Omega_{j}}e^{-\beta N\psi^{(j-1)}_{\lambda_{j}}(E)}\nonumber\\
    &\times\int_{-\infty}^{\infty}e^{\frac{i\beta Np}{2t}}\det^{-\beta/2}(1+ipH^{(j-1)}(E))h_{\beta}(p,E)\diff p\diff E\label{eq:Ef}\\
    &+O(e^{-c\log^{2}N})\nonumber,
\end{align}
where
\begin{align}
    h_{\beta}(p,E)&:=\int_{\mbb{R}^{|I|}}\det^{-\beta/2}\bigl(1-itD(\mbf{k})H^{(j-1)}_{I^{(j-1)}}(p,E)\bigr)\prod_{j=1}^{|I|}\frac{e^{-ik_{j}a_{j}}}{k_{j}-i\epsilon}\frac{\diff k_{j}}{2\pi},\\
    H^{(j-1)}(p,E)&:=\bigl((X^{(j-1)}-E)^{2}+(\eta^{(j-1)}(E))^{2}+ip\bigr)^{-1}.
\end{align}
and $H_{I}$ is the principal submatrix of $H$ indexed by $I$. Note that 
\begin{align}
    H^{(j-1)}(0,E)&=H^{(j-1)}(E)
\end{align}
with $H^{(j-1)}(E)$ defined by \eqref{eq:H}. Note that we have also restricted $E$ to the region
\begin{align}
    \Omega_{j}&:=\Bigl\{E:|E-E_{\lambda_{j}}|<\sqrt{\frac{t}{N}}\log N\Bigr\}
\end{align}
using Lemma \ref{lem:phi2}.

We now give some heuristics for the behaviour of $\mbb{E}^{(\beta)}_{j}[f_{\epsilon}(\mbf{v}_{j})]$. Assume that $h(p,E)$ grows polynomially in $|p|$. As in the proof of Lemma \ref{lem:K}, we can localise $p$ to the region $|p|\lesssim\Bigl(N\Tr{(H^{(j-1)}(E))^{2}}\Bigr)^{-1/2}\log N$. In this region, $\wt{H}_{I}^{(j-1)}(p,E)$ is diagonally dominant and so
\begin{align}
	\det^{-\beta/2}\bigl(1-itD(\mbf{k})H_{I^{(j-1)}}^{(j-1)}(p,E)\bigr)&\simeq\prod_{a=1}^{|I|}\frac{1}{(1-itk_{a}(H^{(j-1)}_{I^{(j-1)}}(p,E))_{aa})^{\beta/2}}.
\end{align}
Since $X$ satisfies the entry-wise local law and $X^{(j)}$ is a projection onto a space of co-dimension $j$, we expect that $X^{(j)}$ also satisfies the entry-wise local law with overwhelming probability, in which case
\begin{align}
	t(H^{(j-1)}_{I^{(j-1)}}(p,E))_{aa}&\simeq 1.
\end{align}
Altogether we obtain
\begin{align}
	\det^{-\beta/2}\bigl(1-itD(\mbf{k})H_{I^{(j-1)}}^{(j-1)}(p,E)\bigr)&\simeq\prod_{a=1}^{|I|}\frac{1}{(1-itk_{a})^{\beta/2}},
\end{align}
and the right-hand side is precisely the characteristic function of $|I|$ independent real or complex standard Gaussians. To summarise, the approximate independence of components of $\mbf{v}_{j}$ (and hence $\mbf{u}_{j}$) is due to two properties:
\begin{enumerate}
	\item the entries of $H^{(j-1)}(p,E)$ are well-approximated by those of $H(p,E)$;
	\item $H(p,E)$ is diagonally dominant.
\end{enumerate}
In the following we make these heuristics precise. 

Let us first study the entries of $H^{(j)}(p,E)$. Let $z=u+i\eta$ with $\textup{dist}(z,[-2,2])\gtrsim t^{2}$ and observe that
\begin{align}
    \bigl((X-u)^{2}+\eta^{2}+ip\bigr)^{-1}&=\frac{G(u+w(p))-G(u-w(p))}{2w(p)}\label{eq:H1},
\end{align}
where
\begin{align}
    w(p)&:=\sqrt{\frac{\sqrt{\eta^{4}+p^{2}}-\eta^{2}}{2}}+i\textup{sgn}(p)\sqrt{\frac{\sqrt{\eta^{4}+p^{2}}+\eta^{2}}{2}}.
\end{align}
Note that $w(0)=i\eta$ and $|\Im w(p)|\simeq |\eta|+\sqrt{|p|}$. By Lemma \ref{lem:resolventProjection}, for any vectors $\mbf{e},\,\mbf{f}\in\mbb{C}^{N-j+1}$,
\begin{align}
    \mbf{e}^{*}G^{(j-1)}(w)\mbf{f}-\mbf{e}^{*}V_{j}G^{(j)}(w)V_{j}^{*}\mbf{f}&=\frac{\mbf{v}_{j}^{*}G^{(j-1)}(w)\mbf{f}\mbf{e}^{*}G^{(j-1)}(w)\mbf{v}_{j}}{\mbf{v}_{j}^{*}G^{(j-1)}(w)\mbf{v}_{j}}.
\end{align}

Applying the concentration of quadratic forms in Lemma \ref{lem:conc}, we obtain
\begin{align}
    \mbf{v}_{j}^{*}G^{(j-1)}(w)\mbf{f}\mbf{e}^{*}G^{(j-1)}(w)\mbf{v}_{j}&=\left[1+O_{\prec}\left(\frac{1}{\sqrt{Nt^{3}}}\right)\right]\frac{t}{N}\mbf{e}^{*}G^{(j-1)}(w)H^{(j-1)}_{\lambda_{j}}G^{(j-1)}(w)\mbf{f}.
\end{align}
Since $\mbf{v}_{j}^{*}G^{(j-1)}(w)\mbf{v}_{j}=\left[1+O_{\prec}\left(\frac{1}{\sqrt{Nt^{3}}}\right)\right]t\Tr{H^{(j-1)}_{\lambda_{j}}G^{(j-1)}(w)}$, we conclude the following.
\begin{lemma}\label{lemma:entrywisePropagation}
We have
\begin{align}
    \mbf{e}^{*}V_{j}G^{(j)}(w)V_{j}^{*}\mbf{f}&=\left[1+O_{\prec}\left(\frac{1}{\sqrt{Nt^{3}}}\right)\right]\mbf{e}^{*}G^{(j-1)}(w)\mbf{f}.
\end{align}
\end{lemma}

As a corollary, we can propagate entry-wise local laws from $X^{(j-1)}$ to $X^{(j)}$.
\begin{corollary}\label{cor:entrywisePropagation}
Let $I\subset[N]$. If
\begin{align}
    \max_{i_{1},i_{2}\in I}\Bigl|\mbf{e}_{i_{1}}^{(j-1)*}G^{(j-1)}(w)\mbf{e}_{i_{2}}^{(j-1)}-M_{i_{1},i_{2}}(w)\Bigr|&\prec\frac{1}{\sqrt{N|\Im w|}},\label{eq:entrywise(j-1)}
\end{align}
then
\begin{align}
    \max_{i_{1},i_{2}\in I}\Bigl|\mbf{e}_{i_{1}}^{(j)*}G^{(j)}(w)\mbf{e}_{i_{2}}^{(j)}-M_{i_{1},i_{2}}(w)\Bigr|&\prec\frac{1}{\sqrt{N|\Im w|}},\label{eq:entrywise(j)}
\end{align}
\end{corollary}
Since \eqref{eq:entrywise(j-1)} is true for $j=1$ by the entry-wise local law for $X$, we can apply Corollary \ref{cor:entrywisePropagation} repeatedly to work on the event that \eqref{eq:entrywise(j)} holds for each $j=1,...,m$.

We now return to \eqref{eq:Ef}. Factorise the Gaussian part from the determinant in the $\mbf{k}$ integral:
\begin{align*}
    \det^{-\beta/2}\bigl(1-itD(\mbf{k})H^{(j)}_{I^{(j)}}(p,E)\bigr)&=\frac{\det^{-\beta/2}\bigl(1-B^{(j)}_{I^{(j)}}(\mbf{k},p,E)\bigr)}{\prod_{a=1}^{|I|}\bigl(1-ik_{a}\bigr)^{\beta/2}},
\end{align*}
where we have defined 
\begin{align}
    \bigl(B^{(j)}_{I^{(j)}}(\mbf{k},p,E)\bigr)_{j_{1},j_{2}}&=\frac{ik_{j_{1}}}{1-ik_{j_{1}}}\Bigl(t\bigl(H^{(j)}_{I^{(j)}}(p,E)\bigr)_{j_{1},j_{2}}-\delta_{j_{1},j_{2}}\Bigr).
\end{align}
For values of $p$ and $E$ that dominate the integral, we can approximate the determinant.
\begin{lemma}\label{lem:detB}
Let $|E-E^{(j-1)}_{\lambda_{j}}|<\sqrt{\frac{t}{N}}\log N$ and $|p|<\frac{\log N}{\sqrt{N\Tr{(H^{(j-1)}_{\lambda_{j}})^{2}}}}$. Then
\begin{align}
	\sup_{\mbf{k}\in\mbb{R}^{|I|}}\|B^{(j)}_{I^{(j)}}(\mbf{k},p,E)\|&\lesssim\frac{\log N}{\sqrt{Nt^{3}}}.
\end{align}
In particular, for any $D>0$ there is an $L>0$ such that
\begin{align}
	\det^{-\beta/2}\Bigl(1-B^{(j)}_{I^{(j)}}(\mbf{k},p,E)\Bigr)&=\exp\Biggl(\frac{\beta}{2}\sum_{m=0}^{L}\frac{1}{m}\tr\bigl(B^{(j)}_{I^{(j)}}(\mbf{k},p,E)\bigr)^{m}\Biggr)+O(N^{-D}),
\end{align}
uniformly in $\mbf{k}\in\mbb{R}^{|I|}$.
\end{lemma}
\begin{proof}
The local law and the definition of $\eta^{(j)}(E)$ imply that
\begin{align}
    t\bigl(H^{(j)}_{I^{(j)}}(p,E)\bigr)_{a,a}&=t\bigl(H^{(j)}_{I^{(j)}}(E)\bigr)_{a,a}+O\Bigl(\frac{\log N}{\sqrt{Nt^{3}}}\Bigr)\\
    &=1+O\Bigl(\frac{\log N}{\sqrt{Nt^{3}}}\Bigr),
\end{align}
and
\begin{align}
    t\bigl(H^{(j)}_{I^{(j)}}(p,E)\bigr)_{a,b}&=O\Bigl(\frac{\log N}{\sqrt{Nt^{3}}}\Bigr),
\end{align}
for $a\neq b$.
\end{proof}

To localise the $p$ integral to the desired region, we need to show that $h_{\beta}(p,E)$ has polynomial growth.
\begin{lemma}\label{lem:hBound}
We have
\begin{align}
    |h_{\beta}(p,E)|&\lesssim \Biggl(1+\frac{|p|}{t^{2}}+\log\epsilon^{-1}\Biggr)^{|I|}.
\end{align}
\end{lemma}
\begin{proof}
We make a different factorisation of the determinant:
\begin{align}
	\det^{-\beta/2}\Bigl(1-itD(\mbf{k})H^{(j)}_{I^{(j)}}(p,E)\Bigr)&=\frac{\det^{-\beta/2}\Bigl(1-\wt{B}^{(j)}_{I^{(j)}}(\mbf{k},p,E)\Bigr)}{\prod_{a=1}^{|I|}(1-itk_{a}\bigl(H^{(j)}_{I^{(j)}}(p,E)\bigr)_{a,a})^{\beta/2}},
\end{align}
where 
\begin{align}
	\bigl(\wt{B}^{(j)}_{I^{(j)}}(\mbf{k},p,E)\bigr)_{a,b}&=\frac{itk_{a}\bigl(H^{(j)}_{I^{(j)}}(p,E)\bigr)_{a,b}}{1-itk_{a}\bigl(H^{(j)}_{I^{(j)}}(p,E)\bigr)_{a,a}}.
\end{align}
We claim that
\begin{align}
	\bigl\|\wt{B}^{(j)}_{I^{(j)}}(\mbf{k},p,E)\bigr\|&\lesssim\frac{1}{\sqrt{Nt^{3}}},\label{eq:Bbound},
\end{align}
uniformly in $p\in\mbb{R}$ and $\mbf{k}\in\mbb{R}^{|I|}$. To prove this, we consider separately the cases $|p|>C$ and $|p|<C$ for some $C>2\|X\|$. In the former case we have
\begin{align}
    \bigl(H^{(j)}_{I^{(j)}}(p,E)\bigr)_{a,b}&=\frac{1}{ip}\Bigl(\delta_{a,b}-\frac{1}{ip}\bigl(|X-z|^{2}\bigr)_{ab}+O\Bigl(\frac{1}{p^{2}}\Bigr)\Bigr),
\end{align}
from which we obtain
\begin{align}
    \bigl|\bigl(\tilde{B}^{(j)}_{I^{(j)}}(\mbf{k},p,E)\bigr)_{a,b}\bigr|&\leq\frac{\bigl|\bigl(H^{(j)}_{I^{(j)}}(p,E)\bigr)_{a,b}}{\bigl|\Re\bigl(H^{(j)}_{I^{(j)}}(p,E)\bigr)_{a,a}\bigr|}\\
    &\prec\frac{1}{\sqrt{N}}.
\end{align}
Note that we used the fact that $\bigl(|X-z|^{2}\bigr)_{ab}\prec N^{-1/2}$ when $a\neq b$ since $X$ satisfies the local law.

When $|p|<C$, by \eqref{eq:H1} and Corollary \ref{cor:entrywisePropagation}, we have
\begin{align}
	\bigl(H^{(j)}_{I^{(j)}}(p,E)\bigr)_{a,a}&=\int\frac{\rho_{sc}(x)\diff x}{|x-z|^{2}+ip}+O_{\prec}\Biggl(\frac{1}{\sqrt{N|\Im w(p)|}|w(p)|}\Biggr),
\end{align}
and so
\begin{align}
	\Re\bigl(H^{(j)}_{I^{(j)}}(p,E)\bigr)_{a,a}&\gtrsim\frac{\sqrt{|\kappa_{u}|+\eta}}{\eta+\sqrt{|p|}}.
\end{align}
Similarly, for $a\neq b$ we have
\begin{align}
	|\bigl(H^{(j)}_{I^{(j)}}(p,E)\bigr)_{a,b}|&\lesssim\frac{1}{\sqrt{N(\eta+\sqrt{|p|})}}\cdot\frac{1}{\eta+\sqrt{|p|}}.
\end{align}
Together these imply \eqref{eq:Bbound}.

Inserting the bound \eqref{eq:Bbound} in the integral over $p$, we find
\begin{align}
	|h_{\beta}(p,E)|&\lesssim\prod_{a}\int_{-\infty}^{\infty}\frac{\diff k}{\bigl|1-itk\bigl(H^{(j)}_{I^{(j)}}(p,E)\bigr)_{a,a}\bigr|^{\beta/2}\sqrt{k^{2}+\epsilon^{2}}}\\
	&\lesssim\Biggl(1+\frac{|p|}{t^{2}}+\log\epsilon^{-1}\Biggr)^{|I|}.
\end{align}
\end{proof}

We can now prove Lemma \ref{lem:gaussianApprox}.
\begin{proof}[Proof of Lemma \ref{lem:gaussianApprox}]
Combining Lemma \ref{lem:hBound} with the upper bound
\begin{align*}
    \bigl|\det^{-1}\bigl(1+ipH^{(j-1)}(E)\bigr)\bigr|&\leq \exp\left\{-\frac{Np^{2}\Tr{(H^{(j-1)}(E))^{2}}}{1+p^{2}\|H^{(j-1)}(E)\|^{2}}\right\},
\end{align*}
we deduce that we can localise $p$ to the region $|p|\leq\frac{\log N}{\sqrt{N\Tr{(H^{(j-1)}(E))^{2}}}}$. In this region we Taylor expand $\det^{-\beta/2}\Bigl(1-B^{(j-1)}_{I^{(j-1)}}(\mbf{k},p,E)\Bigr)$ as in Lemma \ref{lem:detB}. This leads to $\mbf{k}$-integrals of the form
\begin{align}
	\int_{\mbb{R}^{|I|}}\frac{e^{-ix\sum_{j}k_{j}}}{\prod_{j}(1-ik_{j})^{\beta/2}(k_{j}-i\epsilon)}\Biggl[1+P\Biggl(\Biggl\{\frac{k_{j}}{1-ik_{j}}\Biggr\}\Biggr)\Biggr]\frac{\diff\mbf{k}}{(2\pi)^{|I|}},
\end{align}
where $P$ is a polynomial of finite degree whose coefficients are $O_{\prec}\bigl((Nt^{3})^{-1/2}\bigr)$. If $\beta=2$, we can evaluate the integral by the residue theorem to obtain a polynomial in $x$ multiplied by $e^{-|I|x}$. If $\beta=1$, we first write
\begin{align}
	\frac{1}{\sqrt{1-ik}}&=\frac{1}{\sqrt{2\pi}}\int_{-\infty}^{\infty}e^{-\frac{1}{2}(1-ik)\xi^{2}}\diff\xi,
\end{align}
before applying the residue theorem to the $k$-integral.
\end{proof}
   
\section{Proofs of the main results}
\begin{proof}[Proof of Theorem \ref{thm1}]
The result will follow if we can show that the point process
\begin{align}
	\xi&:=\{(d_{E}^{-1}(\lambda_{i}-E),N\|\mbf{u}_{i}\|_{\infty}^{2}-\ell^{(\beta)}_{N}(0)):i\in[N]\}
\end{align}
converges weakly to a point process whose correlation functions factorise into an eigenvalue part (the sine/Airy process) and an eigenvector part which is a Poisson process. Indeed, we can then apply the continuous mapping theorem to the map
\begin{align}
	\xi&\mapsto N|\mbf{u}_{n(E)}\|_{\infty}^{2}-\ell^{(\beta)}_{N}(0),
\end{align}
where
\begin{align}
	n(E)&=\textup{argmin}\{d_{E}^{-1}|\lambda_{n}-E|\}.
\end{align}

By Lemma \ref{lem:momentmatching1}, we know that the point process $\xi$ is universal, so we need only prove the convergence for Gaussian-divisible matrices. It is enough to consider the statistics
\begin{align}
	\mc{L}_{m}(f)&:=\mbb{E}\sum_{i_{1},...,i_{m}}f(d_{E}^{-1}(\lambda_{i_{1}}-E),...,d_{E}^{-1}(\lambda_{i_{m}}-E),N\|\mbf{u}_{i_{1}}\|_{\infty}^{2}-\ell^{(\beta)}_{N}(0),...,N\|\mbf{u}_{i_{m}}\|_{\infty}^{2}-\ell^{(\beta)}_{N}(0)),
\end{align}
for $f\in C^{\infty}_{c}(\mbb{R}^{2m})$. Using the identity
\begin{align}
	f(X)&=\int_{0}^{\infty}f'(x)\mbf{1}_{(x,\infty)}(X)\diff x
\end{align}
for $X>0$, we can consider instead
\begin{align}
	\mbb{E}\sum_{i_{1},...,i_{m}}f(d_{E}^{-1}(\lambda_{i_{1}}-E),...,d_{E}^{-1}(\lambda_{i_{m}}-E))\prod_{a=1}^{m}\mbf{1}_{(x,\infty)}(N\|\mbf{u}_{i_{a}}\|_{\infty}^{2}-\ell^{(\beta)}_{N}(0)).
\end{align}
Finally, using the inclusion-exclusion principle, we can replace $\mbf{1}_{(x,\infty)}(N\|\mbf{u}_{i}\|_{\infty}^{2}-\ell^{(\beta)}_{N}(0))$ with $\prod_{j\in J}\mbf{1}_{(x,\infty)}(N|u_{i,j}|^{2}-\ell^{(\beta)}_{N}(0))$ for $J\subset[N]$ (see \eqref{eq:Pm} and \eqref{eq:Pm2} below). It is therefore sufficient to consider the statistics
\begin{align}
	\mbb{E}\sum_{i_{1},...,i_{m}}f(d_{E}^{-1}(\lambda_{i_{1}}-E),...,d_{E}^{-1}(\lambda_{i_{m}}-E))\prod_{a=1}^{m}\prod_{j\in J_{a}}\mbf{1}_{(x,\infty)}(N|u_{i_{a},j}|^{2}-\ell^{(\beta)}_{N}),
\end{align}
as $J_{a},\,a=1,...,m$ range over finite subsets of $[N]$. By Proposition \ref{prop:gauss-divisible}, we can replace $|u_{i_{a},j}|^{2}$ with $|Z^{(\beta)}_{i_{a},j}|^{2}$ for a collection of i.i.d. standard Gaussians $Z^{(\beta)}_{i,j}$ that are independent of $\{\lambda_{i}\}$. We thus obtain the desired factorisation of the correlation functions. Moreover, since the Gaussian replacements are i.i.d., the eigenvector part converges to a Poisson process with intensity $e^{-\frac{\beta x}{2}}$, which gives the desired Gumbel law.
\end{proof}

Let us prove Theorem \ref{thm3} before sketching the proof of Theorem \ref{thm2}, which is very similar.
\begin{proof}[Proof of Theorem \ref{thm3}]
For $m\in\mbb{N}$, let
\begin{align}
	P_{m}&:=\sum_{n=1}^{m}(-1)^{n-1}\mbb{E}\sum_{\substack{K\subset [N]\times [N]\\|K|=m}}\prod_{(i,j)\in K}\mbf{1}_{(x,\infty)}\bigl(N|u_{i,j}|^{2}-\ell^{(\beta)}_{N}(1)\bigr).\label{eq:Pm}
\end{align}
Then the probability we are interested in is $P_{\infty}$ and we have
\begin{align}
	P_{2m}&\leq P_{\infty}\leq P_{2m+1}.\label{eq:Pm2}
\end{align}
Let $\wt{P}_{m}$ denote the corresponding probability for an independent Gaussian-divisible Wigner matrix $\wt{W}=(1+t)^{-1/2}(X+\sqrt{t}Y)$ that $t$-matches with $W$.

We can rearrange the sum over subsets $K\subset [N]\times [N]$ into sums over subsets $I,\,J\subset[N]$ by considering the number of distinct $i_{k}$ and $j_{k}$ in $\{(i_{k},j_{k}):k=1,...,m\}$. Thus, it is sufficient to prove that
\begin{align}
	P_{x}(J_{1},...,J_{p})&:=\sum_{|I|=p}\mbb{E}\prod_{a=1}^{p}\prod_{j\in J_{a}}\mbf{1}_{(x,\infty)}\bigl(N|u_{i_{a},j}|^{2}-\ell^{(\beta)}_{N}(1)\bigr)
\end{align}
is universal, for any $p\leq m$ and $J_{a}\subset[N],\,a=1,...,p$ such that $\sum_{a}|J_{a}|=m$. We split $P_{x}(J_{1},...,J_{p})$ into two terms as follows:
\begin{align}
	P_{x}(J_{1},...,J_{p})&=A_{x}(J_{1},...,J_{p})+B_{x}(J_{1},...,J_{p}),
\end{align}
where
\begin{align}
	A_{x}(J_{1},...,J_{p})&=\mbb{E}\sum_{|I|=p}\prod_{a=1}^{p}\mbf{1}_{1}(\mc{N}(\hat{I}_{i_{a}}))\prod_{j\in J_{a}}\mbf{1}_{(x,\infty)}\bigl(N|u_{i_{a},j}|^{2}-\ell^{(\beta)}_{N}(1)\bigr).
\end{align}
Thus $A_{x}$ contains the contribution from the event that the intervals $\hat{I}_{i_{a}}$ contain only $\lambda_{i_{a}}$, and $B_{x}$ contains the contribution from the complimentary event. Let $\hat{P}_{x},\,\hat{A}_{x}$ and $\hat{B}_{x}$ denote the quantities obtained by replacing $|u_{i_{a},j}|^{2}$ with $\hat{v}^{2}_{i_{a},j}$. By Lemma \ref{lem:eigapprox} and Proposition \ref{prop:p-gauss-divisible}, we have 
\begin{align}
	B_{x}(J_{1},...,J_{p})&\leq\hat{B}_{x-N^{-\epsilon}}(J_{1},...,J_{p})\lesssim N^{p-m-\delta}.
\end{align}
On the event $\{\mc{N}(\hat{I}_{i_{a}})=1\}$, we have
\begin{align}
	\mbf{1}_{(x+N^{-\epsilon},\infty)}(N\hat{v}^{2}_{i_{a},j}-\ell^{(\beta)}_{N}(1))&\leq\mbf{1}_{(x,\infty)}(N|u_{i_{a},j}|^{2}-\ell^{(\beta)}_{N}(1))\leq\mbf{1}_{(x-N^{-\epsilon},\infty)}(N\hat{v}^{2}_{i_{a},j}-\ell^{(\beta)}_{N}(1)),
\end{align}
and thus
\begin{align}
	\hat{A}_{x+N^{-\epsilon}}&\leq A_{x}\leq\hat{A}_{x-N^{-\epsilon}}.
\end{align}
Altogether we have
\begin{align}
	\hat{P}_{x-N^{-\epsilon}}+O(N^{p-m-\delta})&\leq P_{x}\leq \hat{P}_{x+N^{-\epsilon}}+O(N^{p-m-\delta}).
\end{align}
Since $\hat{P}_{x}$ is universal by Lemma \ref{lem:momentmatching2}, we conclude that $P_{x}$ is also universal. It remains to show that $|u_{i_{a},j}|^{2}$ can be replaced by $|\xi_{i_{a},j}|^{2}$ for a collection of independent Gaussians, which follows from Proposition \ref{prop:gauss-divisible} in the complex case. In the real case, we cannot use Proposition \ref{prop:gauss-divisible} because of the support condition on $f$; instead we reduce to the GOE by imposing the four-moment matching condition on the original Wigner matrix $W$. 
\end{proof}

\begin{proof}[Proof of Theorem 1.2]
Let 
\begin{align}
	\Omega&=\{x:N|x|^{2}>\ell^{(\beta)}_{N}(D)\}.
\end{align}
By inclusion-exclusion, we have
\begin{align}
	A-B&\leq \mbb{P}\Biggl(\max_{\lambda\in I}\|\mbf{u}_{\lambda}\|_{\infty}>\sqrt{\frac{\ell^{(\beta)}_{N}(D)}{N}}\Biggr)\leq A,
\end{align}
where
\begin{align}
	A&=\mbb{E}\sum_{i,j}\mbf{1}_{I}(\lambda_{i})\mbf{1}_{\Omega}(u_{i,j}),
\end{align}
and
\begin{align}
	B&=\mbb{E}\sum_{i_{1}}\sum_{j_{1}<j_{2}}\mbf{1}_{I}(\lambda_{i_{1}})\mbf{1}_{\Omega}(u_{i_{1},j_{1}})\mbf{1}_{\Omega}(u_{i_{1},j_{2}})\\
	&+\mbb{E}\sum_{i_{1}<i_{2}}\sum_{j_{1}}\mbf{1}_{I}(\lambda_{i_{1}})\mbf{1}_{I}(\lambda_{i_{2}})\mbf{1}_{\Omega}(u_{i_{1},j_{1}})\mbf{1}_{\Omega}(u_{i_{2},j_{1}})\\
	&+\mbb{E}\sum_{i_{1}<i_{2}}\sum_{j_{1}<j_{2}}\mbf{1}_{I}(\lambda_{i_{1}})\mbf{1}_{I}(\lambda_{i_{2}})\mbf{1}_{\Omega}(u_{i_{1},j_{1}})\mbf{1}_{\Omega}(u_{i_{2},j_{2}}).
\end{align}
By the argument in the proof of Theorem \ref{thm3}, $A$ and $B$ are universal, so it suffices to estimate them in the Gaussian-divisible case. By Proposition \ref{prop:gauss-divisible}, we have
\begin{align}
    A&\lesssim\mbb{E}\sum_{i}\mbf{1}_{I}(\lambda_{i})\cdot N\mbb{P}\bigl(|Z^{(\beta)}|^{2}>\ell^{(\beta)}_{N}(D)\bigr)\\
    &\lesssim N^{1-D}\rho_{sc}(I).
\end{align}
We bound $B$ from above in a similar way, the only difference being that if $\beta=1$ then we require $|I|\lesssim d_{E}$ in order to apply Proposition \ref{prop:gauss-divisible}. 

To obtain \eqref{eq:thm13}, we use optimal rigidity (see e.g. Tao-Vu \cite[Corollary 15]{tao_random_2013}): for any $D'>0$ there is a constant $C>0$ such that
\begin{align}
    \max_{i\in[N]}\mbb{P}\Bigl(|\lambda_{i}-\gamma_{i}|>\frac{\log^{C}N}{N^{2/3}\hat{i}^{1/3}}\Bigr)\leq N^{-D'}.
\end{align}
The conclusion now follows from \eqref{eq:thm11} by choosing $I$ to be centred at $\gamma_{i}$ with width $N^{-2/3}\hat{i}^{-1/3}$.
\end{proof}

\paragraph{Acknowledgements} This work was mostly carried out while the author was at Queen Mary University of London and supported by the Royal Society grant number RF/ERE/210051. We also acknowledge support from the ERC Advanced Grant ``RMTBeyond" No. 101020331.

\appendix
\section{Proofs for subsection \ref{sec:partialdiag}}
\begin{proof}[Proof of Lemma \ref{lem:phi}]
Recall that $E_{*}$ is defined by
\begin{align}
    t\Tr{G^{2}(E_{*})}&=1.
\end{align}
Using the local law outside the spectrum and the monotonicity of $\Tr{G^{2}(E)}$, we obtain
\begin{align}
    E_{*}&=\frac{2+t}{\sqrt{1+t}}+O\Bigl(\frac{1}{Nt}\Bigr).
\end{align}
Let
\begin{align}
    \lambda_{*}&=E_{*}-t\Tr{G(E_{*})}.
\end{align}
To solve
\begin{align}
    \Re z&=\lambda+t\Re\Tr{G(z)},\label{eq:re_z}\\
    \Im z&=t\Im\Tr{G(z)},\label{eq:im_z}
\end{align}
we first fix $E=\Re z$ and solve the second equation, which is equivalent to \eqref{eq:eta(s)} when its solution is non-zero. This happens for $|\lambda|<\lambda_{*}$.

Consider the deterministic equivalent of \eqref{eq:eta(s)}, namely
\begin{align}
    t\int\frac{\rho_{sc}(x)\diff x}{(x-E)^{2}+\eta^{2}(E)}&=1,
\end{align}
whose solution is 
\begin{align}
    \eta_{\infty}(E)&=\frac{t}{\sqrt{1+t}}\sqrt{1-\frac{(1+t)E^{2}}{(2+t)^{2}}}.
\end{align}
Note that we allow $\eta_{\infty}(E)$ to be imaginary if $E\geq\frac{2+t}{\sqrt{1+t}}$.

Returning to \eqref{eq:eta(s)}, let
\begin{align}
    \Omega_{E}&=\Bigl\{\eta:|\eta^{2}-\eta^{2}_{\infty}(E)|<\frac{\sqrt{|\kappa_{\lambda}|+t^{2}}}{N}\Bigr\}.
\end{align}
For $\eta\in\Omega_{E}$ we have
\begin{align}
    \frac{t}{N}\sum_{n}\frac{1}{(\lambda_{n}-E)^{2}+\eta^{2}}&=t\int\frac{\rho_{sc}(x)\diff x}{(x-E)^{2}+\eta^{2}}+O\Bigl(\frac{1}{Nt(|\kappa_{\lambda}|+t^{2})}\Bigr)\\
    &=1+\int\frac{t(\eta_{\infty}^{2}(E)-\eta^{2})\rho_{sc}(x)\diff x}{((x-E)^{2}+\eta^{2})((x-E)^{2}+\eta^{2}_{\infty}(E))}+O\Bigl(\frac{1}{Nt(|\kappa_{\lambda}|+t^{2})}\Bigr),
\end{align}
from which we obtain
\begin{align}
    |\eta^{2}(E)-\eta^{2}_{\infty}(E)|&\lesssim\frac{\sqrt{|\kappa_{\lambda}|+t^{2}}}{N}.\label{eq:eta(E)approx}
\end{align}
Using this estimate in \eqref{eq:re_z}, we have
\begin{align}
    \frac{E-\lambda}{t}&=\frac{1}{N}\sum_{n}\frac{\lambda_{n}-E}{(\lambda_{n}-E)^{2}+\eta^{2}(E)}\\
    &=\int\frac{(x-E)\rho_{sc}(x)\diff x}{(x-E)^{2}+\eta^{2}(E)}+O\Bigl(\frac{1}{Nt\sqrt{|\kappa_{\lambda}|+t^{2})}}\Bigr)\\
    &=-\frac{E}{2+t}+\int\frac{(\eta^{2}_{\infty}(E)-\eta^{2}(E))(x-E)\rho_{sc}(x)\diff x}{((x-E)^{2}+\eta^{2}(E))((x-E)^{2}+\eta^{2}_{\infty}(E))}+O\Bigl(\frac{1}{Nt\sqrt{|\kappa_{\lambda}|+t^{2})}}\Bigr)\\
    &=-\frac{E}{2+t}+O\Bigl(\frac{1}{Nt\sqrt{|\kappa_{\lambda}|+t^{2})}}\Bigr).
\end{align}
In the third equality we have used the identity
\begin{align}
    \int\frac{(x-E)\rho_{sc}(x)\diff x}{(x-E)^{2}+\eta^{2}_{\infty}(E)}&=-\frac{E}{2+t}.
\end{align}
This implies that
\begin{align}
    E_{\lambda}&=\frac{2+t}{2(1+t)}\lambda+O\Bigl(\frac{1}{N\sqrt{|\kappa_{\lambda}|+t^{2}}}\Bigr),\label{eq:e_lambda}\\
    \eta^{2}_{\lambda}&=\frac{t^{2}}{1+t}\Bigl(1-\frac{\lambda^{2}}{4(1+t)}\Bigr)+O\Bigl(\frac{\sqrt{|\kappa_{\lambda}|+t^{2}}}{N}\Bigr),\label{eq:eta_lambda}
\end{align}
when $|\lambda|<\lambda_{*}$. 

When $|\lambda|\geq\lambda_{*}$, there is a real-valued solution to \eqref{eq:re_z} in the region $|E|\geq E_{*}$. If $0<|\lambda|-\lambda_{*}<ct^{2}$ for some small $c>0$, we rewrite \eqref{eq:re_z} as
\begin{align}
    \lambda-\lambda_{*}+t(E-E_{*})^{2}\Tr{G^{2}(E_{*})G(E)}&=0.
\end{align}
When $|\lambda|-\lambda_{*}>ct^{2}$, we rewrite \eqref{eq:z1} as
\begin{align}
    E&=\lambda+tm_{sc}(E)+O\Bigl(\frac{1}{Nt}\Bigr).
\end{align}
In either case, using the local law outside the spectrum we find
\begin{align}
    E_{\lambda}&=\frac{2+t}{2(1+t)}\lambda+\frac{t}{\sqrt{1+t}}\sqrt{\frac{\lambda^{2}}{4(1+t)}-1}+O\Bigl(\frac{1}{N\sqrt{|\kappa_{\lambda}|+t^{2}}}\Bigr).\label{eq:e_pm}
\end{align}
The estimates in \eqref{eq:z1}-\eqref{eq:z3} follow directly from \eqref{eq:e_lambda}, \eqref{eq:eta_lambda} and \eqref{eq:e_pm}. Since the above argument relied only on the averaged local law, it applies to $X^{(j)}$ by interlacing.

For the final statement, observe that $\Re\phi_{\lambda}(z)$ is even in $\Im z$ so it suffices to consider the case $\Im z>0$. If $s>Ct$ for sufficiently large $C$, then since $|\Tr{G(z)}|\lesssim1$ we have
\begin{align}
    \Re\bigl(\phi_{\lambda}(z_{\lambda}+is)-\phi_{\lambda}\bigr)&=-\int_{0}^{s}\frac{\eta_{\lambda}+r}{t}\bigl(1-t\Tr{|G(z_{\lambda}+ir)|^{2}}\Bigr)\diff r\\
    &\lesssim-\int_{s/2}^{s}\frac{\eta_{\lambda}+r}{t}\bigl(1-t\Tr{|G(z_{\lambda}+ir)|^{2}}\Bigr)\diff r\\
    &=-\int_{s/2}^{s}\Bigl(\frac{\eta_{\lambda}+r}{t}-\Im\Tr{G(z_{\lambda}+ir)}\Bigr)\diff r\\
    &\lesssim-\int_{s/2}^{s}\frac{r}{t}\diff r\\
    &\lesssim-\frac{s^{2}}{t}.
\end{align}
In the second line we used the fact that $t\Tr{|G(z_{\lambda}+ir)|^{2}}\leq1$ for $r>0$ by monotonicity.

Now consider the region $-\epsilon\eta_{\lambda}<s<Ct$. By the local law,
\begin{align}
	\Tr{|G(z_{\lambda}+is)|^{4}}&=\int\frac{\rho_{sc}(x)\diff x}{|x-z_{\lambda}-is|^{4}}+O\Biggl(\frac{1}{N(t\sqrt{|\kappa_{\lambda}|+t^{2}}+s)^{4}}\Biggr)\\
	&\gtrsim\frac{\sqrt{|\kappa_{\lambda}|+t^{2}+s}}{(t\sqrt{|\kappa_{\lambda}|+t^{2}}+s+\max(-\kappa_{\lambda},0))^{3}}.
\end{align}
Inserting this estimate into the formula
\begin{align}
	\Re\bigl(\phi_{\lambda}(z_{\lambda}+is)-\phi_{\lambda}\bigr)&=-2\int_{0}^{s}\int_{0}^{r}(\eta_{\lambda}+r)(\eta_{\lambda}+q)\Tr{|G(z_{\lambda}+iq)|^{4}}\diff q\diff r\nonumber\\
    &-\frac{1-t\Tr{|G(z_{\lambda})|^{2}}}{2t}(s^{2}+2\eta_{\lambda}s),
\end{align}
we obtain the desired result for $s\geq-\epsilon\eta_{\lambda}$ for any small constant $\epsilon>0$. For $s\in[-\eta_{\lambda},-\epsilon\eta_{\lambda}]$, $z_{\lambda}+is$ might be outside the domain of the local law. Instead we use the formula
\begin{align}
	\Re\bigl(\phi_{\lambda}(z_{\lambda}+is)-\phi_{\lambda}\bigr)&=-\int_{0}^{-s}(\eta_{\lambda}-r)\Bigl(\Tr{|G(z_{\lambda}-ir)|^{2}}-\frac{1}{t}\Bigr)\diff r.
\end{align}
Since $\Tr{|G(z_{\lambda}-ir)|^{2}}$ increases monotonically from $t^{-1}$ as $r$ increases from 0, we have
\begin{align}
	\Re\bigl(\phi_{\lambda}(z_{\lambda}+is)-\phi_{\lambda}\bigr)&\leq-\Bigl(\Tr{|G(z_{\lambda}-i\epsilon\eta_{\lambda})|^{2}}-\frac{1}{t}\Bigr)\int_{\epsilon\eta_{\lambda}}^{s}(\eta_{\lambda}-r)\diff r.
\end{align}
We can now use the local law at $z_{\lambda}-i\epsilon\eta_{\lambda}$ to obtain the desired result.
\end{proof}

\begin{proof}[Proof of Lemma \ref{lem:phi2}]
Using the definition of $\eta^{2}(E)$, we can compute
\begin{align}
    \psi'_{\lambda}(E)&=\frac{E-\lambda}{t}-\Tr{(X-E)\bigl((X-E)^{2}+\eta^{2}(E)\bigr)^{-1}},
\end{align}
and
\begin{align}
    \psi''_{\lambda}(E)&=2\eta^{2}(E)\Tr{\bigl((X-E)^{2}+\eta^{2}(E)\bigr)^{-2}}\nonumber\\&+2\partial_{E}(\eta^{2}(E))\Tr{(X-E)\bigl((X-E)^{2}+\eta^{2}(E)\bigr)^{-2}}.
\end{align}
The derivative of $\eta^{2}(E)$ can be obtained by differentiating \eqref{eq:eta(s)}:
\begin{align}
    \partial_{E}(\eta^{2}(E))&=\frac{\Tr{(X-E)\bigl((X-E)^{2}+\eta^{2}(E)\bigr)^{-2}}}{\Tr{\bigl((X-E)^{2}+\eta^{2}(E)\bigr)^{-2}}}.
\end{align}
Thus we have
\begin{align}
    \psi''_{\lambda}(E)&=2\eta^{2}(E)\Tr{\bigl((X-E)^{2}+\eta^{2}(E)\bigr)^{-2}}+\frac{2\Tr{(X-E)\bigl((X-E)^{2}+\eta^{2}(E)\bigr)^{-2}}^{2}}{\Tr{\bigl((X-E)^{2}+\eta^{2}(E)\bigr)^{-2}}}\\
    &\gtrsim\frac{1}{t},
\end{align}
by the local law.
\end{proof}

\begin{proof}[Proof of \ref{lem:conc}]
We only give a sketch since the computations are very similar to those in the proof of Lemma \ref{lem:K}. As in \cite[Lemma 7.2]{maltsev_bulk_2024}, we bound the Laplace transform and use Markov's inequality. Using Lemma \ref{lem:Sduality}, the Laplace transform is given by
\begin{align}
	\mbb{E}^{(\beta)}_{i}e^{x\mbf{v}_{i}^{*}A\mbf{v}_{i}}&=c_{N}\int_{-\infty}^{\infty}e^{-\beta N\psi^{(i-1)}_{\lambda_{i}}(E)}\int_{-\infty}^{\infty}e^{\frac{i\beta Np}{2t}}\det^{-\beta/2}(1+ipB(E))\Bigr)\diff p\diff E,
\end{align}
where
\begin{align}
    B(E)&=\sqrt{H^{(i-1)}(E)}\Bigl(1-\frac{xt}{N}\sqrt{H^{(i-1)}(E)}A\sqrt{H{(i-1)}(E)})^{-1}\Bigr)^{-1}\sqrt{H^{(i-1)}(E)}.
\end{align}
If
\begin{align}
	\frac{|x|t}{N}\bigl\|\sqrt{H^{(i-1)}(E)}A\sqrt{H^{(i-1)}(E)}\bigr\|<1-\delta,
\end{align}
we can treat $A$ as a perturbation and bound the $E$ integral using Lemma \ref{lem:phi2} and the $p$ integral using the bound
\begin{align}
	|\det^{-1}(1+ipB)|&\leq e^{-\frac{p^{2}\tr B^{2}}{2(1+p^{2}\|B\|^{2})}}.
\end{align}
Ultimately we obtain the bound
\begin{align}
	\mbb{E}^{(\beta)}_{i}e^{x\mbf{v}_{i}^{*}A\mbf{v}_{i}-\frac{\beta xt}{2}\Tr{H^{(i-1)}_{\lambda_{i}}A}}&\lesssim e^{-\frac{\beta xt}{2}\Tr{H^{(i-1)}_{\lambda_{i}}A}}\det^{-\beta/2}\Biggl(1-\frac{xt}{N}\sqrt{H^{(i-1)}_{\lambda_{i}}}A\sqrt{H^{(i-1)}_{\lambda_{i}}}\Biggr)\\
	&\leq\exp\Biggl(\frac{\beta^{2}}{4}\cdot\frac{\frac{x^{2}t^{2}}{N}\Tr{(H^{(i-1)}_{\lambda_{i}}A)^{2}}}{1-\frac{xt}{N^{1/2}}\Tr{(H^{(i-1)}_{\lambda_{i}}A)^{2}}^{1/2}}\Biggr),
\end{align}
which we insert into Markov's inequality to obtain \eqref{eq:conc}.
\end{proof}

\section{Proofs for subsection \ref{sec:charpoly}}
\begin{proof}[Proof of Lemma \ref{lem:F_2}]
Let $\Lambda=\diag(\bs\lambda,\bs\lambda)$. By representing the determinants by integrals over anti-commuting variables and then taking the expectation over $Y^{(m)}$, we can derive the formula
\begin{align}
    F_{2}(\bs\lambda,\{\mbf{v}_{j}\})&=\frac{1}{2^{m}}\left(\frac{N}{\pi t}\right)^{2m^{2}}\Delta^{2}(\lambda)\int_{M^{H}_{2m}(\mbb{C})}e^{-\frac{N}{2t}\tr Q^{2}}\nonumber\\&\times\det\bigl(1_{m}\otimes X^{(m)}-(\Lambda+iQ)\otimes1_{N-m}\bigr)\diff Q.
\end{align}
Change variable to $Q\mapsto U\mu U^{*}-i\Lambda$ to obtain
\begin{align}
    F&=\frac{1}{2^{m}}\left(\frac{N}{\pi t}\right)^{2m^{2}}\Delta^{2}(\lambda)\int_{\mbb{R}^{2m}}\Delta^{2}(\mu)I(\mu)\prod_{j=1}^{2m}e^{-\frac{N}{2t}(\mu_{j}^{2}-\Lambda_{j}^{2})}\det(X^{(m)}-i\mu_{j})\diff\mu,
\end{align}
where
\begin{align}
    I(\mu)&=\int_{U(2m)/U(1)^{2m}}e^{-\frac{iN}{t}\tr U\mu U^{*}\Lambda}\diff U.
\end{align}
Here we have deformed the $\mu_{j}$-contours back to the real axis.

The integral over $U$ is the limit of the HCIZ integral as $\lambda_{m+i}\to\lambda_{i}$ and can be performed by l'H\^{o}pital's rule:
\begin{align}
    \int_{U(2m)}e^{-\frac{iN}{t}\tr U\mu U^{*}\Lambda}\diff U&=\left(\frac{t}{N}\right)^{2m(m-1)}\frac{1}{\Delta(\mu)\Delta^{4}(\lambda)}\det\begin{pmatrix}e^{\frac{iN}{t}\mu_{j}\lambda_{k}}&\mu_{j}e^{\frac{iN}{t}\mu_{j}\lambda_{k}}\end{pmatrix}.
\end{align} 
Inserting this above we find
\begin{align}
    F&=c_{m}\left(\frac{N}{t}\right)^{2m}\frac{1}{\Delta^{2}(\lambda)}\int_{\mbb{R}^{2m}}\Delta(\mu)\det\begin{pmatrix} e^{-\frac{N}{2t}(\mu_{j}+i\lambda_{k})^{2}}&\mu_{j}e^{-\frac{N}{2t}(\mu_{j}+i\lambda_{k})^{2}}\end{pmatrix}\prod_{j=1}^{2m}\det(X^{(m)}-i\mu_{j})\diff\mu.
\end{align}
Expanding the determinant and using the symmetry of the integrand, we obtain \eqref{eq:F}.
\end{proof}

\begin{proof}[Proof of Lemma \ref{lem:F_1,2}]
For ease of notation we drop the superscript $(n)$. Representing the determinants by Gaussian integrals, we have
\begin{align}
    F_{1,2}(\bs\lambda,\{\mbf{v}_{j\leq n}\})&=\left(\frac{N}{2\pi t}\right)^{N/2}\lim_{\epsilon\to0}\int\exp\left\{\tr\Psi^{*}X\Psi-\tr\Lambda\Psi^{*}\Psi-\frac{N}{2t}\mbf{x}^{T}\bigl((X-\lambda_{1})^{2}+\epsilon\bigr)\mbf{x}\right\}\nonumber\\
    &\times\mbb{E}\exp\left\{-\frac{N}{2}\mbf{x}^{T}Y^{2}\mbf{x}-N\tr(B_{1}+B_{2})Y\right\}\diff\Psi\diff\mbf{x},
\end{align}
where
\begin{align}
    B_{1}&:=\frac{1}{2\sqrt{t}}\left(\mbf{x}\mbf{x}^{T}(X-\lambda_{1})+(X-\lambda_{1})\mbf{x}\mbf{x}^{T}\right),\\
    B_{2}&:=\frac{\sqrt{t}}{2N}\left(\Psi\Psi^{*}-\bar{\Psi}\Psi^{T}\right)
\end{align}
Let $A>0$ be real symmetric and observe the following identity
\begin{align}
    I_{n}(A,B)&:=\frac{1}{2^{n/2}}\left(\frac{N}{\pi}\right)^{n^{2}/2}\int_{\textup{Sym(n)}}e^{-\frac{N}{4}\tr AY^{2}-N\tr BY}\diff Y\nonumber\\
    &=\prod_{i}a_{i}^{-1/2}\prod_{i<j}\left(\frac{a_{i}+a_{j}}{2}\right)^{-1/2}\exp\left\{2N\textup{vec}(B)^{T}(1\otimes A+A\otimes 1)^{-1}\textup{vec}(B)\right\},
\end{align}
where $a_{i}$ are the eigenvalues of $A$. If $A=1_{n}$ we obtain $e^{N\tr B^{2}}$ as in the usual Gaussian integral. In our case, $A=1+2\mbf{x}\mbf{x}^{T}$,
\begin{align}
	2(1\otimes A+A\otimes1)^{-1}&=1-\frac{r}{1+r}(\mbf{u}\mbf{u}^{T}\otimes1+1\otimes\mbf{u}\mbf{u}^{T})+\frac{2r^{2}}{(1+r)(1+2r)}\mbf{u}\mbf{u}^{T}\otimes\mbf{u}\mbf{u}^{T},
\end{align}
and so
\begin{align}
    I_{n}(A,B)&=\frac{1}{(1+2r)^{1/2}(1+r)^{(n-1)/2}}\exp\left\{N\left(\tr B^{2}-\frac{2r\mbf{u}^{T}B^{2}\mbf{u}}{1+r}+\frac{2(r\mbf{u}^{T}B\mbf{u})^{2}}{(1+r)(1+2r)}\right)\right\},
\end{align}
where we have written $\mbf{x}=\sqrt{r}\mbf{u}$ for $\mbf{u}\in S^{n-1}$. Inserting $B=B_{1}$ we find
\begin{align*}
    \tr B_{1}^{2}-\frac{2r\mbf{u}^{T}B^{2}_{1}\mbf{u}}{1+r}+\frac{2(r\mbf{u}^{T}B_{1}\mbf{u})^{2}}{(1+r)(1+2r)}&=\frac{r^{2}}{2t(1+r)}\left[\mbf{u}^{T}(X-\lambda_{1})^{2}\mbf{u}+\frac{(\mbf{u}^{T}(X-\lambda_{1})\mbf{u})^{2}}{1+2r}\right].
\end{align*}
For the other terms we use the identity
\begin{align*}
    \tr B^{2}-\frac{2r\mbf{u}^{T}B^{2}\mbf{u}}{1+r}+\frac{2(r\mbf{u}^{T}B\mbf{u})^{2}}{(1+r)(1+2r)}&=\tr\Bigl(B\bigl(1-\frac{r}{1+r}\mbf{u}\mbf{u}^{T}\bigr)\Bigr)^{2}+\frac{(r\mbf{u}^{T}B\mbf{u})^{2}}{(1+r)^{2}(1+2r)}
\end{align*}
to obtain
\begin{align*}
    \tr B^{2}_{2}-\frac{2r\mbf{u}^{T}B^{2}_{2}\mbf{u}}{1+r}+\frac{2(r\mbf{u}^{T}B_{2}\mbf{u})^{2}}{(1+r)(1+2r)}&=-\frac{t}{2N^{2}}\tr\Bigl(\Psi^{*}\bigl(1-\frac{r}{1+r}\mbf{u}\mbf{u}^{T}\bigr)\Psi\Bigr)^{2}\\
    &+\frac{t}{2N^{2}}\tr\Psi^{*}\bigl(1-\frac{r}{1+r}\mbf{u}\mbf{u}^{T}\bigr)\bar{\Psi}\Psi^{T}\bigl(1-\frac{r}{1+r}\mbf{u}\mbf{u}^{T}\bigr)\Psi\\
    &+\frac{tr^{2}(\mbf{u}^{T}\Psi\Psi^{*}\mbf{u})^{2}}{N^{2}(1+r)^{2}(1+2r)}
\end{align*}
and
\begin{align}
    &2\tr B_{1}B_{2}-\frac{2r\mbf{u}^{T}(B_{1}B_{2}+B_{2}B_{1})\mbf{u}}{1+r}+\frac{4r^{2}\mbf{u}^{T}B_{1}\mbf{u}\mbf{u}^{T}B_{2}\mbf{u}}{(1+r)(1+2r)}\\
    &=-\frac{1}{N}\tr\Psi^{*}\Bigl(\frac{r}{1+r}(\mbf{u}\mbf{u}^{T}(X-\lambda_{1})+(X-\lambda_{1})\mbf{u}\mbf{u}^{T})-\frac{2r^{2}\mbf{u}^{T}(X-\lambda_{1})\mbf{u}}{(1+r)(1+2r)}\mbf{u}\mbf{u}^{T}\Bigr)\Psi.
\end{align}

For the quartic terms in $\Psi$ we introduce auxiliary integrals:
\begin{align}
	e^{-\frac{t}{2N}\tr\Bigl(\Psi^{*}\bigl(1-\frac{r}{1+r}\mbf{u}\mbf{u}^{T}\bigr)\Psi\Bigr)^{2}}&=\frac{1}{2^{m/2}}\Biggl(\frac{N}{\pi t}\Biggr)^{m^{2}/2}\int_{\textup{Herm}_{\mbb{C}}(m)}e^{-\frac{N}{2t}\tr|B|^{2}}\nonumber\\
	&\times e^{\frac{i}{2}\tr\Bigl(B\Psi^{*}\bigl(1-\frac{r}{1+r}\mbf{u}\mbf{u}^{T}\bigr)\Psi-B^{T}\Psi^{T}\bigl(1-\frac{r}{1+r}\mbf{u}\mbf{u}^{T}\bigr)\bar{\Psi}\Bigr)}\diff B,
\end{align}
and
\begin{align}
	e^{\frac{t}{2N}\tr\Bigl(\Psi^{*}\bigl(1-\frac{r}{1+r}\mbf{u}\mbf{u}^{T}\bigr)\bar{\Psi}\Psi^{T}\bigl(1-\frac{r}{1+r}\mbf{u}\mbf{u}^{T}\bigr)\Psi}&=\Biggl(\frac{N}{\pi t}\Biggr)^{m(m-1)/2}\int_{\textup{Skew}_{\mbb{C}}(m)}e^{-\frac{N}{2t}\tr|D|^{2}}\nonumber\\
	&\times e^{\frac{1}{2}\tr\Bigl(D\Psi^{*}\bigl(1-\frac{r}{1+r}\mbf{u}\mbf{u}^{T}\bigr)\bar{\Psi}-D^{*}\Psi^{T}\bigl(1-\frac{r}{1+r}\mbf{u}\mbf{u}^{T}\bigr)\Psi\Bigr)}\diff D,
\end{align}
and
\begin{align}
	e^{\frac{tr^{2}}{N(1+r)^{2}(1+2r)}(\mbf{u}^{T}\Psi\Psi^{*}\mbf{u})^{2}}&=\sqrt{\frac{Nr^{2}}{2\pi t(1+r)^{2}(1+2r)}}\int_{-\infty}^{\infty}e^{-\frac{Nr^{2}}{4t(1+r)^{2}(1+2r)}x^{2}+\frac{r^{2}x\mbf{u}^{T}\Psi\Psi^{*}\mbf{u}}{(1+r)^{2}(1+2r)}}\diff x.
\end{align}

The dependence of the exponent on the anticommuting variables is quadratic; integrating these out we obtain
\begin{align}
    \pf M(S,x),
\end{align}
where
\begin{align}
	M&=S\otimes\Biggl(1-\frac{r}{1+r}\mbf{u}\mbf{u}^{T}\Biggr)-\begin{pmatrix}0&\Lambda-\lambda_{1}\\-\Lambda+\lambda_{1}&0\end{pmatrix}\otimes1_{N}+1_{2m}\otimes\begin{pmatrix}0&\mathcal{X}\\-\mathcal{X}^{T}&0\end{pmatrix}\nonumber\\
	&+\frac{xr^{2}}{(1+r)^{2}(1+2r)}J_{2m}\otimes\mbf{u}\mbf{u}^{T},\label{eq:M}
\end{align}
with
\begin{align}
	J_{2m}&=\begin{pmatrix}0&1_{m}\\-1_{m}&0\end{pmatrix},\\
	S&=\begin{pmatrix}D&iB\\-iB^{T}&D^{*}\end{pmatrix},\\
	\mathcal{X}&=X-\lambda_{1}-\frac{r}{1+r}\Bigl(\mbf{u}\mbf{u}^{T}(X-\lambda_{1})+(X-\lambda_{1})\mbf{u}\mbf{u}^{T}\Bigr)+\frac{2r^{2}\mbf{u}^{T}(X-\lambda_{1})\mbf{u}}{(1+r)(1+2r)}\mbf{u}\mbf{u}^{T}.
\end{align}
By making a block decomposition of $M$ with respect to the projection $1_{2m}\otimes\mbf{u}\mbf{u}^{T}$ and making the change of variables $S=UJ(\mbf{s})U^{T}-J(\bs\lambda)$, we obtain \eqref{eq:F_1,2}. Note that we can put the limit $\epsilon\to0$ inside the integral since the integrand at $\epsilon=0$ is integrable as a function of $r$. To see this, we write
\begin{align}
    B(\mbf{s},r)&=\wt{B}\Bigl(\mbf{s},\frac{r}{1+r},\frac{r^{2}}{(1+r)(1+2r)},r\Bigr),
\end{align}
where
\begin{align}
    \wt{B}(\mbf{s},a,b,c)&=iaJ(\mbf{s})+b(\mbf{u}^{T}(X^{(n)}-\lambda_{1})\mbf{u}+x)J-cU^{*}J(\wt{\bs\lambda}-\lambda_{1})U.
\end{align}
Using the pfaffian addition formula
\begin{align}
   \pf(A+B)&=\sum_{q=0}^{n}\sum_{|I|=2q}(-1)^{\sum_{i\in I}i-q}\pf(A_{I})\pf(B_{I^{c}}),
\end{align}
and the fact that $\pf(U^{*}(\wt{\bs{\lambda}}-\lambda_{1})U)=0$, we deduce that the degree of $c=r$ in $p(x,U,\mbf{s},r)$ is at most $n-1$. Thus the integrand decays as $r^{-2}$ for large $r$ and is integrable in $r$.

For illustration, consider the case $\Lambda=\lambda_{1}1_{m}$. Then the second term in \eqref{eq:M} disappears and we can diagonalise $S$ by a symplectic matrix, 
\begin{align}
	S&=U\begin{pmatrix}0&-i\mbf{s}\\i\mbf{s}&0\end{pmatrix}U^{T},
\end{align}
so that
\begin{align}
	\pf M(B,D,x)&=\prod_{j=1}^{m}\det\Bigl(\mathcal{X}-is_{j}\Bigl(1-\frac{r}{1+r}\mbf{u}\mbf{u}^{T}\Bigr)+\frac{xr^{2}}{(1+r)^{2}(1+2r)}\mbf{u}\mbf{u}^{T}\Bigr)\\
	&=\prod_{j=1}^{m}\frac{\det(X-\lambda_{1}-is_{j})}{(1+r)^{2}}\Biggl[1+\frac{r^{2}(\mbf{u}^{T}(X-\lambda_{1})\mbf{u}+x)-ir(1+2r)s_{j}}{1+2r}\mbf{u}^{T}G(\lambda_{1}-is_{j})\mbf{u}\Biggr]\\
    &=\prod_{j=1}^{m}\frac{\det(X^{(n)}-\lambda_{1}-is_{j})}{1+r}\Biggl[\frac{1}{1+r}\bigl(\mbf{u}^{T}G(\lambda_{1}-is_{j})\mbf{u}\bigr)^{-1}-\frac{irs_{j}}{1+r}\nonumber\\&+\frac{r^{2}}{(1+r)(1+2r)}\bigl(\mbf{u}^{T}(X^{(n)}-\lambda_{1})\mbf{u}+x\bigr)\Biggr].
\end{align}
We recognise the factor in square brackets in the last line as $p(x,U,\mbf{s},r)$ in the limit $\lambda_{j}=\lambda_{1}$.
\end{proof}

\begin{proof}[Proof of Lemma \ref{lem:K(r)}]
The proof is very similar to that of Lemma \ref{lem:K}, the only difference being that we generalise $\eta(E)$ and $H(E)$ to also depend on $r$. Recall the definition \eqref{eq:eta(E,r)} of $\eta^{(n)}(E,r)$ and note that $\eta^{(n)}(E,\infty)=\eta^{(n)}(E)$, where $\eta^{(n)}(E)$ is defined in \eqref{eq:eta(s)}. By the duality formula in Lemma \ref{lem:Sduality}, we have
\begin{align}
	K_{n}(r,\lambda)&=\int_{-\infty}^{\infty}e^{-N\psi_{\lambda}(E,r)}\int_{-\infty}^{\infty}e^{\frac{iNrp}{2t(1+r)}}\det^{-1/2}\Bigl(1+ipH^{(n)}(E,r)\Bigr)\diff p\diff x,
\end{align}
for any $\eta\in\mbb{R}$, where
\begin{align}
	\psi^{(n)}_{\lambda}(E,r)&=\frac{1}{2t}\Biggl((E-\lambda)^{2}-\frac{r}{1+r}(\eta^{(n)}(E,r))^{2}\Biggr)+\frac{1}{2}\Tr{\log((X^{(n)}-E)^{2}+(\eta^{(n)}(E,r))^{2})},
\end{align}
and
\begin{align}
    H^{(n)}(E,r)&=\bigl((X^{(n)}-E)^{2}+(\eta^{(n)}(E,r))^{2}\bigr)^{-1}.
\end{align}

Define $\psi^{(n)}_{\lambda,r}:=\psi_{\lambda}(E^{(n)}_{\lambda,r},r)$. By arguing as for Lemma \ref{lem:phi2}, we can obtain
\begin{align}
    \psi_{\lambda}(E,r)-\psi_{\lambda,r}&\gtrsim\frac{t(1+r)}{r}(E-E^{(n)}_{\lambda,r})^{2}.
\end{align}
Now the analysis proceeds exactly as in the proof of Lemma \ref{lem:K}.
\end{proof}

\section{Proofs for subsection \ref{sec:E_j}}
\begin{proof}[Proof of Lemma \ref{lem:E_j}]
Let us first rewrite the exponent as follows:
\begin{align}
	\frac{N}{t}\Bigl(\mbf{v}_{j}^{*}(X^{(j-1)})^{2}\mbf{v}_{j}-\lambda_{j}\mbf{v}_{j}^{*}X^{(j-1)}\mbf{v}_{j}+\frac{1}{2}\lambda_{j}^{2}-\frac{1}{2}\bigl(\mbf{v}_{j}^{*}X^{(j-1)}\mbf{v}_{j}\bigr)^{2}\Bigr).
\end{align}
For the quartic term we introduce an auxiliary Gaussian integral:
\begin{align}
	e^{\frac{N}{2t}\bigl(\mbf{v}_{j}^{*}X^{(j-1)}\mbf{v}_{j}\bigr)^{2}}&=\Biggl(\frac{N}{2\pi t}\Biggr)^{1/2}\int_{-\infty}^{\infty}e^{-\frac{N}{2t}E^{2}-\frac{N}{t}E\mbf{v}_{j}^{*}X^{(j-1)}\mbf{v}_{j}}\diff E.
\end{align}
The exponent becomes
\begin{align}
	&\frac{N}{t}\Bigl(\mbf{v}_{j}^{*}(X^{(j-1)})^{2}\mbf{v}_{j}-(\lambda_{j}+E)\mbf{v}_{j}^{*}X^{(j-1)}\mbf{v}_{j}+\frac{1}{2}(\lambda_{j}^{2}+E^{2})\Bigr)\\
	&=\frac{N}{t}\Bigl(\mbf{v}_{j}^{*}\bigl(X^{(j-1)}-\frac{\lambda_{j}+E}{2}\bigr)^{2}\mbf{v}_{j}+\frac{1}{4}(E-\lambda_{j})^{2}\Bigr).
\end{align}
Shifting $E\mapsto \lambda+2E$, we obtain for the expectation
\begin{align}
	\mbb{E}_{j}\bigl[f(\mbf{v}_{j})\bigr]&=\frac{\sqrt{2}}{K_{j}}\Biggl(\frac{N}{\pi t}\Biggr)^{N-j+1/2}\int_{-\infty}^{\infty}e^{-\frac{N}{t}E^{2}}\nonumber\\
    &\times\int_{S^{N-j}}f(\mbf{v}_{j})e^{-\frac{N}{t}\mbf{v}_{j}^{*}\bigl(X^{(j-1)}-\lambda_{j}-E\bigr)^{2}\mbf{v}_{j}}\diff S_{N-j}(\mbf{v}_{j})\diff E.
\end{align}
Applying the duality formula in Lemma \ref{lem:Sduality}, we obtain \eqref{eq:E_j}. In order to restrict $E$ to a suitable region, we first bound $f(\mbf{v}_{j})$ by $\|f\|_{\infty}$ before applying the duality formula.
\end{proof}

\begin{proof}[Proof of Lemma \ref{lem:halfSpace}]
Recall the Fourier representation of $e^{-\epsilon(x-a)}1_{(a,\infty)}(x)$:
\begin{align*}
    e^{-\epsilon(x-a)}1_{(a,\infty)}(x)&=\int_{\mbb{R}}\frac{e^{ik(x-a)}}{k-i\epsilon}\frac{\diff k}{2\pi}.
\end{align*}
We need to justify changing the order of integration in
\begin{align*}
    \int_{\mbb{C}^{n}}\int_{\mbb{R}^{|J|}}e^{-\mbf{x}^{*}A\mbf{x}+i\sum_{j\in J}k_{j}x_{j}}\prod_{j\in J}\frac{e^{-ik_{j}a_{j}}}{k_{j}-i\epsilon}\frac{\diff k_{j}}{2\pi}\diff\mbf{x}
\end{align*}
We cannot use Fubini's theorem since the integral is not absolutely convergent. Instead we use the distributional identity
\begin{align*}
    1_{(a,b)}(x)&=\lim_{\epsilon\to0}\frac{1}{\sqrt{2\pi\epsilon}}\int_{a}^{b}e^{-\frac{1}{2\epsilon}(x-y)^{2}}\diff y.
\end{align*}
Note that although both sides are well-defined functions, the identity does not hold pointwise: the right hand side is not zero when $x=a$ or $x=b$. Since the function in the limit on the right hand side is uniformly bounded in $\epsilon$, we can take the limit outside the integral over $\mbb{C}^{n}$. Now we use the identity
\begin{align*}
    \frac{1}{\sqrt{2\pi\epsilon}}e^{-\frac{1}{2\epsilon}(x-y)^{2}}&=\frac{1}{2\pi}\int_{\mbb{R}}e^{-\frac{\epsilon}{2}k^{2}+ik(x-y)}\diff k.
\end{align*}
At this point we can freely interchange the order of integration, so we intergrate first over $\mbb{C}^{n}$ then $y_{j},\,j\in[|J|]$. Finally, we put the limit $\epsilon\to0$ inside the resulting integral over $\mbf{k}$ (which we can do since the integrand is now absolutely integrable due to the factor $\det^{-1}(1-iD(\mbf{k})A)$) to obtain \eqref{eq:halfSpace}.
\end{proof}

\bibliographystyle{plain}

\end{document}